\documentclass[article]{amsart}
\usepackage{amssymb,amsfonts,amsmath,amsthm}
\usepackage[all,arc]{xy}
\usepackage{enumerate}
\usepackage{mathrsfs}
\usepackage[toc,page]{appendix}
\usepackage[left=3cm, right=3cm, bottom=3cm]{geometry}
\usepackage{graphicx}
\usepackage{tabularx}
\usepackage{multirow} 
\usepackage{caption} 
\usepackage{url}
\usepackage{color}
\usepackage{tikz-cd}  
\usepackage{mathdots} 
\usepackage{romannum} 
\usepackage{hyperref} 
\usepackage{float} 

\newtheorem{thm}{Theorem}[section]

\newtheorem*{thm*}{Theorem}
\newtheorem*{cor*}{Corollary}
\newtheorem*{prop*}{Proposition}
\newtheorem{cor}[thm]{Corollary}
\newtheorem{prop}[thm]{Proposition}
\newtheorem{lem}[thm]{Lemma}

\theoremstyle{definition}
\newtheorem{defn}[thm]{Definition}

\newtheorem{notn}[thm]{Notation}
\newtheorem*{notn*}{Notation}

\theoremstyle{remark}
\newtheorem{rem}[thm]{Remark}

\newtheorem*{idea*}{Idea}

\makeatletter
\let\c@equation\c@thm
\makeatother
\numberwithin{thm}{section}
\numberwithin{equation}{section}

\title[Meromorphic Parahoric Higgs Torsors and Filtered Stokes G-Local Systems]{Meromorphic Parahoric Higgs Torsors and Filtered Stokes G-Local Systems on Curves}

\author{Pengfei Huang and Hao Sun}

\begin{document}
\pagenumbering{arabic}
\maketitle
\begin{abstract}
In this paper, we consider the wild nonabelian Hodge correspondence for principal $G$-bundles on curves, where $G$ is a connected complex reductive group. We establish the correspondence under a ``very good" condition on the irregular type of the meromorphic $G$-connections introduced by Boalch, and thus confirm a conjecture in \cite[\S 1.5]{Boalch2018}. We first give a version of Kobayashi--Hitchin correspondence, which induces a one-to-one correspondence between stable meromorphic parahoric Higgs torsors of degree zero (Dolbeault side) and stable meromorphic parahoric connections of degree zero (de Rham side). Then, by introducing a notion of stability condition on filtered Stokes $G$-local systems, we prove a one-to-one correspondence between stable meromorphic parahoric connections of degree zero (de Rham side) and stable filtered Stokes $G$-local systems of degree zero (Betti side). When $G={\rm GL}_n(\mathbb{C})$, the main result in this paper reduces to that in \cite{BB}.
\end{abstract}
	
\flushbottom
	
	

\renewcommand{\thefootnote}{\fnsymbol{footnote}}
\footnotetext[2]{MSC2020 Class: 14H99, 14L15, 32Q26}

\section{Introduction}

\subsection{Background}

The nonabelian Hodge correspondence (NAHC) aims at giving a correspondence among three objects: Higgs bundles (Dolbeault side), integrable connections (de Rham side) and representations of the fundamental group (Betti side). The relation between Higgs bundles and integrable connections is established with the help of harmonic metrics, while integrable connections and representations of the fundamental group are related by the Riemann--Hilbert correspondence. The NAHC on smooth projective varieties over $\mathbb{C}$ was established by a series of seminal works \cite{Corlette,Donaldson,Hit87,Simp1988,Simp-naH,Simp Higgs Local,Simp2,Simp3}; more details from various aspects can be found in the survey papers \cite{GR2015,Huang2020,Li2019}.

In the case of punctured curves, the study of the NAHC can be distinguished into two main situations depending on the order of poles (under gauge transformations) of meromorphic connections at the punctures. In that regard, a connection is called regular singular if the order of the pole at each puncture is at most one, and irregular singular otherwise. These two algebraic conditions can be characterized analytically by the growth rate of the eigenvalues of the corresponding Higgs fields around the punctures, which leads to two analytic notions, namely, tameness and wildness respectively. By introducing tame harmonic bundles, Simpson established a one-to-one correspondence between stable filtered regular Higgs bundles of degree zero, stable filtered regular integrable connections of degree zero and stable filtered local systems of degree zero \cite{Simp}. This version of the correspondence is usually called the tame NAHC on punctured curves, and the filtered objects of Simpson are adequately characterized by a weighted flag on the fiber of the vector bundle over each puncture; this structure is commonly referred to as a parabolic structure and is due to Seshadri \cite{Seshadri}.

The problem becomes more subtle in the case of general complex reductive structure groups $G$ for the bundle, since integrable $G$-connections with regular singularities cannot be fully classified by the parabolic subgroups of $G$ as in the ${\rm GL}_n(\mathbb{C})$-case described by Simpson. By introducing parahoric objects, Boalch overcame the obstruction and gave a complete description of the Riemann--Hilbert correspondence for complex reductive groups under the tameness condition, which is called
the tame parahoric Riemann-Hilbert correspondence \cite{Bo}. This correspondence was recently extended in \cite{HKSZ} to a full one-to-one tame parahoric NAHC for any complex reductive group, as a correspondence between well-defined moduli spaces (see \cite{KSZ2parh} and \cite{HS} for the construction of the moduli spaces involved). The correspondence was built based on the analytic description of Biquard--Garc\'{i}a-Prada--Mundet i Riera \cite{BGM}, where the authors introduced $G$-Higgs bundles and $G$-connections equipped with appropriate parabolic structures and studied the correspondence within this framework.

When considering irregular singular connections for the structure group ${\rm GL}_n(\mathbb{C})$, a wild NAHC was established through combined works of Sabbah \cite{Sab} and Biquard--Boalch \cite{BB}, involving objects that are ``very good'', a condition formally formulated by Boalch in \cite[Definition 4]{Boalch2018}. This condition means that the meromorphic connections and Higgs fields have diagonal polar parts (up to gauge transformations). In further detail, Sabbah has constructed a model metric used to prove a version of the Kobayashi--Hitchin correspondence for integrable connections with irregular singularities. Biquard--Boalch then studied the Kobayashi--Hitchin correspondence for meromorphic Higgs bundles. A higher dimensional generalization was obtained by Mochizuki \cite{Mochi2011,Mochi2021}, who established a full categorical correspondence among stable good filtered Higgs bundles satisfying a certain vanishing condition, stable good filtered $\lambda$-flat bundles satisfying a vanishing condition, and good wild harmonic bundles. However, a moduli correspondence for these objects is still missing.

In the present work, we deal with the problem of generalizing the theory towards obtaining a wild NAHC for a general complex reductive group $G$. In \cite[Appendix]{Boalch2014}, Boalch showed that there is a one-to-one correspondence between meromorphic $G$-connections with irregular singularities and Stokes $G$-local systems on curves, which can be regarded as a version of Riemann--Hilbert correspondence for principal bundles in the wild case. Recently, the authors of \cite{BBMY} started from affine Springer fibers to construct the moduli stacks of $G$-Higgs bundles, $G$-connections, and Stokes $G$-local systems on $\mathbb{P}^{1}$ with level structures at 0 and $\infty$, and prescribed irregular part at $\infty$. An enhanced Riemann--Hilbert map was defined in this sense.

We summarize the known versions of the NAHC on noncompact curves as follows.

\begin{table}[h!]
\centering
\caption*{Table: NAHC on Noncompact Curves}
	\begin{tabular}{|c|c|c|}
		\hline
		\rule{0pt}{2.7ex} Structure group  & Tame NAHC & Wild NAHC  \\[0.6ex]
		\hline
		\rule{0pt}{2.7ex} $\mathrm{GL}_n(\mathbb{C})$  & Simpson \cite{Simp}  & Sabbah \cite{Sab}, Biquard--Boalch \cite{BB} \\[0.6ex]
		\hline
       \multirow{2}{*}{\rule{0pt}{3ex} $G$} & \rule{0pt}{2.6ex} Biquad--Garc\'ia-Prada--Mundet i Riera \cite{BGM}, & \\[0.5ex] 
       & Boalch \cite{Bo}, HKSZ \cite{HKSZ} &   \\[0.5ex]
		\hline
	\end{tabular}
\end{table}

\noindent In this paper, we consider the wild NAHC for principal bundles on curves, and we prove that there is a one-to-one correspondence between stable meromorphic parahoric Higgs torsors and stable filtered Stokes local systems satisfying a ``very good" condition introduced by Boalch \cite[Definition 4]{Boalch2018}, which is also referring to the unramified case in the reference \cite{BBMY}. 

\subsection{Main Result}

Let $G$ be a connected complex reductive group with a maximal torus $T$, and let $\mathfrak{g}$ and $\mathfrak{t}$ be the corresponding Lie algebras. A real weight is a co-character with coefficients in $\mathbb{R}$, which is called a weight in this paper for simplicity. Let $X$ be a smooth projective curve over $\mathbb{C}$ with a given reduced effective divisor $\boldsymbol{D}$, which is also regarded as a set of distinct points on $X$. We introduce the following notations:
\begin{itemize}
	\item $\boldsymbol\bullet=\{\bullet_x, x \in \boldsymbol D\}$ is a collection of weights labelled by points in $\boldsymbol D$, where $\bullet=\alpha,\beta,\gamma$.
	\item $\boldsymbol{Q}=\{ Q_x, x \in \boldsymbol D\}$ and $\widetilde{\boldsymbol{Q}}=\{ \widetilde{Q}_x, x \in \boldsymbol D\}$ are two collections of irregular types labelled by points in $\boldsymbol{D}$. Here an irregular type at each $x\in\boldsymbol{D}$ refers to an element 
    \begin{align*}
        Q_x\in\mathfrak{t}(K)/\mathfrak{t}(R),
    \end{align*}
    where $K:=\mathbb{C}(\!(z)\!)$ and $R:=\mathbb{C}[\![z]\!]$ under a choice of local coordinate $z$ that vanishes at $x$;
	\item $\mathcal{G}_{\boldsymbol\bullet}$ is the corresponding parahoric group scheme, where $\bullet=\alpha,\beta,\gamma$;
	\item $\varphi_{\boldsymbol\alpha},\nabla_{\boldsymbol\beta}$ are collections of elements in $\mathfrak{g}$, which are regarded as residues of Higgs fields or connections at punctures;
	\item $M_{\boldsymbol\gamma}$ is a collection of elements in $G$, which are regarded as formal monodromies of filtered Stokes $G$-local systems around punctures.
\end{itemize}
Recall that the word \emph{logahoric} is a blend of \emph{logarithmic} and \emph{parahoric}, which is introduced in \cite{Bo}. With the same idea, we introduce a new word: \emph{merohoric}, which is a blend of \emph{meromorphic} and \emph{parahoric}. A \emph{merohoric Higgs torsor} is a meromorphic parahoric Higgs torsor, and more precisely, it is a parahoric torsor equipped with a meromorphic Higgs field. Similarly, a \emph{merohoric connection} is a parahoric torsor equipped with a meromorphic connection (Definition \ref{defn_alg_meroh}). Let $X_{\boldsymbol{Q}}$ be the \emph{irregular curve} determined by the irregular types $\boldsymbol{Q}$ (see \cite[\S 8]{Boalch2014} or \S\ref{subsect_prel_stokes} for a more detailed description). In this paper, we fix irregular types $\boldsymbol{Q}$ and study the correspondence under the assumption that the \emph{merohoric connections are with irregular types $\boldsymbol{Q}$} (see Definition \ref{defn_irregular_type_parh} and Proposition \ref{prop_canonical_form_parah}).   Now we define the following categories (Dolbeault, de Rham and Betti, respectively) with the corresponding notions of stability condition defined within the paper:
\begin{itemize}
	\item $\mathcal{C}_{\rm Dol}(X,\mathcal{G}_{\boldsymbol\alpha}, \varphi_{\boldsymbol\alpha},\widetilde{\boldsymbol{Q}})$: the category of $R$-stable merohoric $\mathcal{G}_{\boldsymbol\alpha}$-Higgs torsors of degree zero on $X$ with irregular types $\widetilde{\boldsymbol{Q}}$ and the Levi factors of residues of the Higgs field at punctures are $\varphi_{\boldsymbol\alpha}$;
	\item $\mathcal{C}_{\rm dR}(X,\mathcal{G}_{\boldsymbol\beta},\nabla_{\boldsymbol\beta},\boldsymbol{Q})$: the category of $R$-stable merohoric $\mathcal{G}_{\boldsymbol\beta}$-connections of degree zero on $X$ with irregular types $\boldsymbol{Q}$ and the Levi factors of residues of the connection at punctures are $\nabla_{\boldsymbol\beta}$;
	\item $\mathcal{C}_{\rm B}(X_{\boldsymbol{Q}},G,\boldsymbol{\gamma}, M_{\boldsymbol\gamma} )$: the category of $R$-stable $\boldsymbol\gamma$-filtered Stokes $G$-local systems on $X_{\boldsymbol{Q}}$ of degree zero and the Levi factors of formal monodromies are $M_{\boldsymbol\gamma}$.  
\end{itemize}
In order to streamline notation, the local data for the objects in these categories is grouped in a table as follows:
\begin{center}
	\begin{tabular}{|c|c|c|c|}
		\hline
		\rule{0pt}{2.6ex} & \ \ \ Dolbeault\ \ \ & \ \ \ de Rham\ \ \ & \ \ \ Betti\ \ \ \\[0.5ex]
		\hline
		\rule{0pt}{2.6ex} weights & $\alpha_x$ & $\beta_x$ & $\gamma_x$ \\[0.5ex]
		\hline
		\rule{0pt}{2.6ex} residues $\backslash$ monodromies & $\varphi_{\alpha_x}$ & $\nabla_{\beta_x}$ & $M_{\gamma_x}$ \\[0.7ex]
		\hline
		\rule{0pt}{2.6ex} irregular types & $\widetilde{Q}_x$ & $Q_x$ & $Q_x$ \\[0.5ex]
		\hline
	\end{tabular}
\end{center}
Let $\nabla_{\beta_x} = s_{\beta_x} + Y_{\beta_x}$ be the Jordan decomposition, where $s_{\beta_x}$ is the semisimple part and $Y_{\beta_x}$ is the nilpotent part. We complete $Y_{\beta_x}$ into an $\mathfrak{sl}_2$-triple $(X_{\beta_x}, H_{\beta_x}, Y_{\beta_x})$. With respect to the relations in the following table,
\begin{center}
	\begin{tabular}{|c|c|c|c|}
		\hline
		\rule{0pt}{2.6ex} & Dolbeault & \ \ de Rham\ \ & Betti \\[0.5ex]
		\hline
		\rule{0pt}{2.6ex} weights & $\frac{1}{2}(s_{\beta_x} + \bar{s}_{\beta_x})$ & $\beta_x$ & $\beta_x - \frac{1}{2}(s_{\beta_x} + \bar{s}_{\beta_x})$ \\[0.7ex]
		\hline
		\rule{0pt}{2.6ex} residues $\backslash$ monodromies & $\frac{1}{2}(s_{\beta_x} - \beta_x) + (Y_{\beta_x} - H_{\beta_x} + X_{\beta_x})$ & $s_{\beta_x} + Y_{\beta_x}$ & $\exp(-2\pi \sqrt{-1}(s_{\beta_x} + Y_{\beta_x}))$ \\[0.7ex]
		\hline
		\rule{0pt}{2.6ex} irregular types & $\frac{1}{2}Q_x$ & $Q_x$ & $Q_x$ \\[0.7ex]
		\hline
	\end{tabular}
\end{center}
the main theorem in this paper is given as follows:
\begin{thm*}[Theorem \ref{thm_alg_dR_Dol} and Theorem \ref{thm_dR_Betti}]
	The categories 
	\begin{align*}
		\mathcal{C}_{\rm Dol}(X,\mathcal{G}_{\boldsymbol\alpha}, \varphi_{\boldsymbol\alpha}, \widetilde{\boldsymbol{Q}} ), \quad \mathcal{C}_{\rm dR}(X,\mathcal{G}_{\boldsymbol\beta}, \nabla_{\boldsymbol\beta}, \boldsymbol{Q} ), \quad \mathcal{C}_{\rm B}(X_{\boldsymbol{Q}},G,\boldsymbol{\gamma}, M_{\boldsymbol\gamma} )
	\end{align*}
	are equivalent.
\end{thm*}

The equivalence of $\mathcal{C}_{\rm Dol}(X,\mathcal{G}_{\boldsymbol\alpha}, \varphi_{\boldsymbol\alpha}, \widetilde{\boldsymbol{Q}} )$ and $\mathcal{C}_{\rm dR}(X,\mathcal{G}_{\boldsymbol\beta}, \nabla_{\boldsymbol\beta}, \boldsymbol{Q} )$ is given in \S\ref{sect_Dol_dR} (Theorem \ref{thm_alg_dR_Dol}). This equivalence comes from a version of Kobayashi--Hitchin correspondence (Theorem \ref{thm_KH_corr}), which is proved by constructing an appropriate model metric in \S\ref{subsect_local_study}. The equivalence of $\mathcal{C}_{\rm dR}(X,\mathcal{G}_{\boldsymbol\beta}, \nabla_{\boldsymbol\beta}, \boldsymbol{Q} )$ and $\mathcal{C}_{\rm B}(X_{\boldsymbol{Q}},G,\boldsymbol{\gamma}, M_{\boldsymbol\gamma} )$ is a generalization of the correspondence between meromorphic $G$-connections and Stokes $G$-local systems (\cite[Theorem A.3 and Corollary A.4]{Boalch2014}). We generalize the correspondence to merohoric connections and filtered Stokes $G$-local systems in \S\ref{subsect_fil_stokes}. By defining an appropriate stability condition on filtered Stokes $G$-local systems (Definition \ref{defn_stab_betti}), we prove the equivalence between the de Rham and Betti categories (Theorem \ref{thm_dR_Betti}).

\subsection{Further Remarks}

Firstly, we consider the case of vector bundles, i.e. $G={\rm GL}_n(\mathbb{C})$. In \cite{BB}, Biquard--Boalch studied the wild NAHC for vector bundles on curves, where they gave a correspondence between very good meromorphic Higgs bundles and very good integrable connections with irregular singularities. In the case of ${\rm GL}_n(\mathbb{C})$, the $R$-stability condition for parahoric torsors is equivalent to the stability condition for parabolic bundles \cite[\S 5]{KSZ2parh}. Thus, the result in \S\ref{sect_Dol_dR} covers the result \cite[Theorem 6.1]{BB} in the categorical version, and the local data is exactly the same. Furthermore, as a special case, the result in this paper relates meromorphic Higgs bundles and filtered Stokes local systems on curves by taking $G={\rm GL}_n(\mathbb{C})$. 

Secondly, due to the lack of a construction of the corresponding moduli spaces, the main result in this paper is only given for categories. Despite this, the Dolbeault and de Rham moduli spaces can be constructed similarly as in \cite{KSZ2parh}. For the Betti side, Boalch constructed the moduli space of Stokes $G$-local systems as an affine GIT quotient \cite[Theorem 9.3]{Boalch2014}, of which the stable points are exactly the irreducible representations. As a special case, if we consider filtered Stokes $G$-local systems with trivial weights, then filtered Stokes $G$-local systems reduce to Stokes $G$-local systems and Boalch's result does give the construction of the corresponding Betti moduli space (Lemma \ref{lem_stokes_trivial}). However, this result does not give the construction of the Betti moduli space considered in this paper in general because an $R$-stable filtered Stokes $G$-local system of degree zero is not necessarily irreducible. Even in the tame case, a stable filtered $G$-local system may not be an irreducible representation \cite{HKSZ,Simp}. The construction of the Betti moduli spaces for filtered local systems, filtered $G$-local systems, and filtered Stokes $G$-local systems with general weights was given by the authors recently in \cite{HS}.

Thirdly, the meromorphic connections considered in this paper are supposed to be gauge equivalent to ones in canonical form (see Definition \ref{defn_cano_form} for canonical forms, and Definition \ref{defn_irregular_type}, Definition \ref{defn_irregular_type_parh} and Proposition \ref{prop_canonical_form_parah} for the setup). This is an important assumption for meromorphic connections in this paper, which satisfies the ``very good" condition in \cite[Definition 4]{Boalch2018}. Although this assumption is not very strong, there exist meromorphic connections which are not gauge equivalent to ones in canonical form. Furthermore, it is proved that any meromorphic $G$-connection is gauge equivalent to one in canonical form on an appropriate cover (see \cite[\S9.5 Theorem]{BaVa} or \cite[Theorem 4.3]{Herr}). Therefore, a complete description of the wild NAHC for principal bundles on curves is expected to be given in equivariant objects, for instance, equivariant meromorphic $G$-Higgs bundles, equivariant meromorphic $G$-connections and equivariant filtered Stokes $G$-local systems.

Finally, if the irregular types $\boldsymbol{Q}$ are trivial, i.e. $Q_x=0$ for any $x \in \boldsymbol{D}$, the main result in this paper reduces to a special tame case. More precisely, the special case requires that logarithmic connections $d+B(z)\frac{dz}{z}$ are gauge equivalent to ones of the form $d+ B_0 \frac{dz}{z}$, where $B(z) \in \mathfrak{g}(\mathbb{C}[[z]])$ and $B_0 \in \mathfrak{g}$ because we only study the unramified case in this paper (see Remark \ref{rem_diff_tame_wild}). In general, given a logarithmic connection $d+B(z)\frac{dz}{z}$, it is gauge equivalent to one $d + B_0 \frac{dw}{w}$ on an appropriate cover $z=w^d$ (see \cite{Herr} for instance). Moreover, Boalch introduced parahoric objects to give a precise description of this phenomenon and establish the Riemann--Hilbert correspondence in the tame case \cite{Bo}.

\vspace{2mm}

\textbf{Acknowledgments}.
We would like to thank G. Kydonakis and L. Zhao for many helpful discussions. We thank G. Kydonakis for his very helpful comments and suggestions on a very early version of this paper. We thank P. Boalch and C. Simpson for enlightening communications. We thank the anonymous referee for very helpful suggestions and comments. P. Huang would like to express the deep gratitude to the Institut des Hautes Études Scientifiques in Paris, France for its kind hospitality during the preparation of this paper, and acknowledges funding by the Deutsche Forschungsgemeinschaft (DFG, German Research Foundation) – Project-ID 281071066 – TRR 191. H. Sun is partially supported by National
Key R$\&$D Program of China (No. 2022YFA1006600) and NSFC (No. 12101243).

\vspace{2mm}

\section{Preliminaries}\label{sect_Pre}

In this section, we briefly review the main algebraic objects and analytic objects studied in this paper. For the algebraic part, we introduce the terminology of \emph{merohoric $\mathcal{G}_{\boldsymbol\theta}$-Higgs torsors} (resp. \emph{merohoric $\mathcal{G}_{\boldsymbol\theta}$-connections}), where \emph{merohoric} is a blend word of \emph{meromorphic} and \emph{parahoric} and $\mathcal{G}_{\boldsymbol\theta}$ is a parahoric group scheme, and introduce stability conditions (see Definitions \ref{defn_alg_meroh} and \ref{defn_alg_stab}). We refer the reader to \cite{BS,HKSZ,KSZ2parh} for a detailed overview of the parahoric notions used in this paper. For the analytic part, we follow the setup in \cite[\S 4]{HKSZ} and give the definition of \emph{metrized $(\boldsymbol\theta,\boldsymbol{n})$-adapted $G$-Higgs bundles} (resp. \emph{$G$-connections}) (Definition \ref{defn_ana_adapted}), which is regarded as the definition of \emph{wild $G$-Higgs bundles} (resp. \emph{$G$-connections}), and the corresponding stability conditions (Definition \ref{defn_ana_stab}). At the end of this section, we introduce two functors
\begin{align*}
\Xi_{\rm Dol} & : \mathcal{C}_{\rm Dol}(X_{\boldsymbol{D}},G,\boldsymbol\theta,\boldsymbol{n}) \rightarrow \mathcal{C}_{\rm Dol}(X,\mathcal{G}_{\boldsymbol\theta},\boldsymbol{n}) \\
\Xi_{\rm dR} & : \mathcal{C}_{\rm dR}(X_{\boldsymbol{D}},G,\boldsymbol\theta,\boldsymbol{n}) \rightarrow \mathcal{C}_{\rm dR}(X,\mathcal{G}_{\boldsymbol\theta},\boldsymbol{n})  ,
\end{align*}
which relate analytic objects and algebraic objects for Higgs bundles (\emph{Dolbeault}) and connections (\emph{de Rham}).

\subsection{Algebraic Part}
Let $G$ be a connected complex reductive group with a maximal torus $T$. Denote respectively by $\mathfrak{g}$ and $\mathfrak{t}$ the corresponding Lie algebras. Let $\mathcal{R}$ be the set of roots. A \emph{real weight} $\theta$ is an element in ${\rm Hom}(\mathbb{G}_m,T) \otimes_{\mathbb{Z}} \mathbb{R}$ such that $r(\theta) \leq 1$ for any root $r \in \mathcal{R}$, where $r(\theta):=\langle \theta, r \rangle$ is the natural pairing of characters and co-characters. In this paper, we call it a \emph{weight} for simplicity. When $r(\theta)<1$ for any $r \in \mathcal{R}$, it is called a \emph{small weight}. Given a weight $\theta$, we define the integer for each root $r \in \mathcal{R}$
\begin{align*}
m_r(\theta):=\lceil -r(\theta) \rceil,
\end{align*}
where $\lceil \cdot \rceil$ is the ceiling function. Then, the \emph{parahoric subgroup $G_{\theta}(K)$} of $G(K)$ is defined as
\begin{align*}
G_{\theta}(K):=\langle T(R), U_r(z^{m_r(\theta)}R), r \in \mathcal{R} \rangle,
\end{align*}
where $R:=\mathbb{C}[\![z]\!]$ and $K:=\mathbb{C}(\!(z)\!)$. An equivalent definition of $G_\theta(K)$ is
\begin{align*}
G_{\theta}(K) =\{g(z) \in G(K) \text{ }|\text{ } z^{\theta} g(z) z^{-\theta} \text{ has a limit as $z \rightarrow 0$}\},
\end{align*}
where $z^{\theta}:=e^{\theta \ln z}$ (see \cite[Lemma 2.2]{HKSZ}). Denote by $\mathcal{G}_{\theta}$ the corresponding group scheme of $G_{\theta}(K)$, which is called the \emph{parahoric group scheme}. Furthermore, 
\begin{align*}
\mathfrak{g}_\theta(K) : =\{ g(z) \in \mathfrak{g}(K) \text{ } | \text{ } z^{\theta} g(z) z^{-\theta} \text{ has a limit as $z \rightarrow 0$}\}
\end{align*}
is regarded as the Lie algebra of $G_\theta(K)$.

Now let $X$ be a smooth projective curve over $\mathbb{C}$, and let $\boldsymbol{D}$ be a fixed reduced effective divisor on $X$, which is also regarded as a collection of distinct points. Denote by $X_{\boldsymbol D}: = X \backslash \boldsymbol{D}$ the punctured (noncompact) curve. Let ${\boldsymbol\theta}=\{\theta_x, x \in \boldsymbol{D}\}$ be a collection of weights indexed by points in $\boldsymbol D$. By gluing the following local data
\begin{align*}
\mathcal{G}_{\boldsymbol\theta}|_{X_{\boldsymbol D}} \cong G \times X_{\boldsymbol D}, \quad \mathcal{G}_{\boldsymbol\theta}|_{\mathbb{D}_x} \cong \mathcal{G}_{\theta_x}, x \in \boldsymbol{D},
\end{align*}
where $\mathbb{D}_x$ is a formal disc around $x$, we obtain a smooth group scheme $\mathcal{G}_{\boldsymbol\theta}$ over $X$. This group scheme $\mathcal{G}_{\boldsymbol\theta}$ is called a \emph{parahoric (Bruhat--Tits) group scheme}. 

Recall that a parabolic subgroup $P$ of $G$ with Lie algebra $\mathfrak{p}$ is determined by a subset of roots $\mathcal{R}_P \subseteq \mathcal{R}$. We define the corresponding parahoric subgroup as follows
\begin{align*}
P_\theta(K) := \langle T(R), U_r(z^{m_r(\theta)}R), r \in \mathcal{R}_P \rangle.
\end{align*}
Denote by $\mathcal{P}_\theta$ the corresponding group scheme. For the global picture, we define the group scheme $\mathcal{P}_{\boldsymbol\theta}$ over $X$ by gluing the local data similarly
\begin{align*}
\mathcal{P}_{\boldsymbol\theta}|_{\mathbb{D}_x} \cong P \times X_{\boldsymbol D}, \quad \mathcal{P}_{\boldsymbol\theta}|_{\mathbb{D}_x} \cong \mathcal{P}_{\theta_x}, x \in \boldsymbol D.
\end{align*}

\begin{notn}\label{notn_parab_levi}
	Given a weight $\theta$, a natural parabolic subgroup is determined which we denote by $P_\theta \subseteq G$. In this paper, we use the notation $L_\theta$ for the corresponding Levi subgroup. The corresponding Lie algebras are denoted by $\mathfrak{p}_{\theta}$ and $\mathfrak{l}_\theta$, respectively. For the construction of $P_\theta(K)$ above, we fix both a weight $\theta$ and a parabolic subgroup $P$, and then define a parahoric subgroup $P_\theta(K) \subseteq P(K)$ and its group scheme $\mathcal{P}_\theta$. If we denote the Levi subgroup (resp. Lie subalgebra) of $P$ (resp. $\mathfrak{p}$) by $L$ (resp. $\mathfrak{l}$), we can define the corresponding parahoric subgroup (resp. Lie subalgebra) $L_\theta(K)$ (resp. $\mathfrak{l}_\theta(K)$) and group scheme $\mathcal{L}_\theta$ similarly. 
\end{notn}

Now let $\mathcal{E}$ be a parahoric $\mathcal{G}_{\boldsymbol\theta}$-torsor on $X$. Let $\varsigma \in \Gamma(X,\mathcal{E}/\mathcal{P}_{\boldsymbol\theta})$ be a section, which is called a \emph{reduction of structure group}. Denote by $\mathcal{E}_{\varsigma}$ the pullback of the diagram
\begin{center}
	\begin{tikzcd}
	\mathcal{E}_{\varsigma} \arrow[r, dotted] \arrow[d, dotted] & \mathcal{E} \arrow[d]\\
	X \arrow[r, "\varsigma"]  & \mathcal{E}/\mathcal{P}_{\boldsymbol\theta}.
	\end{tikzcd}
\end{center}
Then, the pullback $\mathcal{E}_{\varsigma}$ is a parahoric $\mathcal{P}_{\boldsymbol\theta}$-torsor on $X$. Let $\kappa: \mathcal{P}_{\boldsymbol\theta} \rightarrow \mathbb{G}_m$ be a morphism of group schemes over $X$, and we call it a \emph{character} of $\mathcal{P}_{\boldsymbol\theta}$. Furthermore, there is a one-to-one correspondence between characters of $\mathcal{P}_{\boldsymbol\theta}$ and characters of $P$ \cite[Lemma 4.2]{KSZ2parh}:
\begin{align*}
{\rm Hom}(\mathcal{P}_{\boldsymbol\theta},\mathbb{G}_m) = {\rm Hom}(P,\mathbb{C}^*).
\end{align*}
Denote by $\chi: P \rightarrow \mathbb{C}^*$ the corresponding character of $\kappa$. We define the following inner product
\begin{align*}
\langle \theta, \kappa \rangle := \langle \theta, \chi \rangle,
\end{align*}
for any weight $\theta$. 

We return to the parahoric $\mathcal{P}_{\boldsymbol\theta}$-torsor $\mathcal{E}_{\varsigma}$. Given a character $\kappa$, there is a line bundle $\kappa_* (\mathcal{E}_{\varsigma})$ on $X$, and denote it by $L(\varsigma,\kappa)$. As a special case, let $P=G$. If $\varsigma: X \rightarrow \mathcal{E} / \mathcal{G}_{\boldsymbol\theta}$ is trivial, then the pushforward $L(\kappa):=\kappa_* \mathcal{E}$ directly gives a line bundle on $X$. In \cite{KSZ2parh}, the authors introduced a notion of (algebraic) degree of parahoric torsors as follows:
\begin{defn}\label{defn_alg_deg}
	The \emph{parahoric degree} of a parahoric $\mathcal{G}_{\boldsymbol\theta}$-torsor $\mathcal{E}$ with respect to a given reduction $\varsigma$ and a character $\kappa$ is defined as follows
	\begin{align*}
	parh\deg \mathcal{E}(\varsigma,\kappa)=\deg L(\varsigma,\kappa)+ \langle  \boldsymbol\theta, \kappa \rangle,
	\end{align*}
	where $\langle  \boldsymbol\theta, \kappa \rangle:=\sum_{x\in \boldsymbol{D}} \langle  \theta_x, \kappa \rangle$. If $\varsigma: X \rightarrow \mathcal{E} / \mathcal{G}_{\boldsymbol\theta}$ is trivial, we define
	\begin{align*}
	parh\deg \mathcal{E}(\kappa)=\deg L(\kappa)+ \langle  \boldsymbol\theta, \kappa \rangle.
	\end{align*}
	A parahoric $\mathcal{G}_{\boldsymbol\theta}$-torsor $\mathcal{E}$ is of \emph{degree zero}, if we have $parh\deg \mathcal{E}(\kappa)=0$ for any character $\kappa$.
\end{defn}
The definition of a $G$-bundle of \emph{degree zero} comes from the \emph{$R_\mu$-stability condition}, and we refer the reader to \cite[\S 5]{KSZ2parh} and \cite[\S 4.2]{BGM} for more details. Following Ramanathan's definition of stability conditions \cite{Rama1975,Rama19961,Rama19962} for principle bundles, a notion of stability condition is given for parahoric torsors, which is called the $R$-stability condition. We refer the reader to \cite{HKSZ,KSZ2parh} for more details.

Now we move to Higgs bundles and connections. Let $\boldsymbol{n}:=\{n_x, x \in \boldsymbol{D}\}$ be a set of positive integers, and denote by $\boldsymbol{nD}:=\sum_{x \in \boldsymbol{D}} n_x \cdot x$ the corresponding divisor.
\begin{defn}\label{defn_alg_meroh}
	A \emph{meromorphic parahoric $\mathcal{G}_{\boldsymbol\theta}$-Higgs torsor} on $X$ is a pair $(\mathcal{E},\varphi)$, where
	\begin{itemize}
		\item $\mathcal{E}$ is a parahoric $\mathcal{G}_{\boldsymbol\theta}$-torsor on $X$;
		\item $\varphi \in H^0(X, \mathcal{E}(\mathfrak{g}) \otimes K_X(\boldsymbol{nD}))$ is a section, where $\mathcal{E}(\mathfrak{g})$ is the adjoint bundle.
	\end{itemize}
	The section $\varphi$ is called a \emph{meromorphic Higgs field}. Such a pair is called a \emph{merohoric $\mathcal{G}_{\boldsymbol\theta}$-Higgs torsor} for simplicity.
	
	Let $\mathcal{V}$ be a parahoric $\mathcal{G}_{\boldsymbol\theta}$-torsor. A \emph{meromorphic connection} on $\mathcal{V}$ is an integrable meromorphic connection $\nabla:\mathcal{O}_{\mathcal{V}} \rightarrow \mathcal{O}_{\mathcal{V}} \otimes K_X(\boldsymbol{nD})$ compatible with the co-multiplication map of $\mathcal{V}$. Such a pair $(\mathcal{V},\nabla)$ is called a \emph{merohoric $\mathcal{G}_{\boldsymbol\theta}$-connection}. We refer the reader to \cite[Appendix]{ChenZhu} and \cite[Appendix]{HKSZ} for more details about the definition of connections on torsors.
\end{defn}

\begin{rem}\label{rem_mero_ana_to_alg}
	Let $\varphi$ be a meromorphic Higgs field on $\mathcal{E}$. For a puncture $x \in \boldsymbol{D}$, the restriction of $\varphi$ to $\mathbb{D}_x$ can be regarded as an element in $\mathfrak{g}_{\theta_x}(K)dz/z^{n_x}$, and we use the same notation $\varphi(z)$ for the corresponding element. Thus, by the definition of $\mathfrak{g}_{\theta_x}(K)$, it is easy to check that
	\begin{align*}
	z^{\theta_x} \cdot (z^{n_x} \varphi(z)) \cdot z^{-\theta_x} \text{ is bounded as $z$ approaches zero}.
	\end{align*}
	Furthermore, if an element $\varphi(z) \in \mathfrak{g}(K)dz$ satisfies this condition, then $\varphi(z) \in \mathfrak{g}_{\theta_x}(K)dz/z^{n_x}$.
\end{rem}
Before we state the stability condition for merohoric Higgs torsors and merohoric connections, we first define \emph{compatible reductions of structure group}. Given a merohoric $\mathcal{G}_{\boldsymbol\theta}$-Higgs torsor $(\mathcal{E},\varphi)$, a reduction of structure group $\varsigma$ is said to be \emph{$\varphi$-compatible}, if $\varphi$ induces a section (Higgs field) $\varphi_\varsigma : X \rightarrow \mathcal{E}_\varsigma(\mathfrak{p}) \otimes K(\boldsymbol{nD})$. Similarly, let $(\mathcal{V},\nabla)$ be a merohoric $\mathcal{G}_{\boldsymbol\theta}$-connection, a reduction $\varsigma$ is said to be \emph{$\nabla$-compatible}, if $\nabla$ induces a meromorphic connection $\nabla_\varsigma$ on $\mathcal{V}_\varsigma$. If there is no ambiguity, we say \emph{compatible reductions of structure group} for simplicity.

\begin{defn}\label{defn_alg_stab}
	A merohoric $\mathcal{G}_{\boldsymbol\theta}$-Higgs torsor $(\mathcal{E},\varphi)$ is called \emph{$R$-stable} (resp. \emph{$R$-semistable}), if for
	\begin{itemize}
		\item any proper parabolic group $P \subseteq G$,
		\item any compatible reduction of structure group $\varsigma: X \rightarrow \mathcal{E}/\mathcal{P}_{\boldsymbol\theta}$,
		\item any nontrivial anti-dominant character $\kappa: \mathcal{P}_{\boldsymbol\theta} \rightarrow \mathbb{G}_m$, which is trivial on the center of $\mathcal{P}_{\boldsymbol\theta}$,
	\end{itemize}
	one has
	\begin{align*}
	parh\deg \mathcal{E}(\varsigma,\kappa) > 0, \quad (\text{resp. } \geq 0).
	\end{align*}
	
	A merohoric $\mathcal{G}_{\boldsymbol\theta}$-connection $(\mathcal{V},\nabla)$ is called \emph{$R$-stable} (resp. \emph{$R$-semistable}), if for
	\begin{itemize}
		\item any proper parabolic group $P \subseteq G$,
		\item any compatible reduction of structure group $\varsigma: X \rightarrow \mathcal{V}/\mathcal{P}_{\boldsymbol\theta}$,
		\item any nontrivial anti-dominant character $\kappa: \mathcal{P}_{\boldsymbol\theta} \rightarrow \mathbb{G}_m$, which is trivial on the center of $\mathcal{P}_{\boldsymbol\theta}$,
	\end{itemize}
	one has
	\begin{align*}
	parh\deg \mathcal{V}(\varsigma,\kappa) > 0, \quad (\text{resp. } \geq 0).
	\end{align*}	
\end{defn}

\subsection{Analytic Part}\label{subsect_prel_ana}

In this subsection, we review the definition and some properties of metrized $G$-Higgs bundles and metrized $G$-connections on $X_{\boldsymbol D}$. We refer the reader to \cite[\S 3 and \S 4]{HKSZ} for more details. Let $(E,\partial''_E,h)$ be a metrized $G$-bundle on $X_{\boldsymbol D}$, where $\partial''_E$ is a holomorphic structure of $E$ and $h$ is a hermitian metric on $E$. In this paper, we consider a special class of hermitian metrics, which are called \emph{adapted metrics}. 

\begin{defn}[Definition 2.12 in \cite{BGM} or Definition 3.2 in \cite{HKSZ}]
	If a metric $h$ takes the form $h_0 \cdot |z|^{2\theta} e^c$ such that
	\begin{itemize}
		\item $h_0$ is the standard constant metric;
		\item $\theta \in {\rm Hom}(\mathbb{G}_m,T) \otimes_{\mathbb{Z}} \mathbb{R}$ is a weight;
		\item ${\rm Ad}(|z|^{2 \theta_x})c = o (\ln |z|)$ and the curvature $F_h$ of the associated connection to $h$ is in $L^1$,
	\end{itemize}
	then $h$ is called a \emph{$\theta$-adapted metric} or \emph{adapted metric} for simplicity.
\end{defn}
\noindent A metric is always assumed to be an adapted metric from now on. Now let $P$ be a parabolic subgroup of $G$. We take a holomorphic section $\sigma: X_{\boldsymbol D} \rightarrow E/P$, and obtain a $P$-bundle $E_\sigma$ from the diagram
\begin{center}
	\begin{tikzcd}
	E_{\sigma} \arrow[r, dotted] \arrow[d, dotted] & E \arrow[d]\\
	X_{\boldsymbol D} \arrow[r, "\sigma"]  & E/P.
	\end{tikzcd}
\end{center}
Furthermore, a holomorphic structure $\partial''_{E_\sigma}$ and a metric $h_\sigma$ are also induced from the diagram, and thus, we obtain a metrized $P$-bundle $(E_\sigma,\partial''_{E_\sigma}, h_\sigma)$. 
\begin{defn}\label{defn_ana_degree}
	Taking a character $\chi: P \rightarrow \mathbb{C}^*$, the \emph{analytic degree} of $(E,\partial''_E,h)$ is defined as
	\begin{equation*}
	\deg^{\rm an} E(h, \sigma, \chi):= \frac{\sqrt{-1}}{2\pi} \int_{X_{\boldsymbol{D}}} \chi_* F_{h_\sigma}.
	\end{equation*}
	If $\sigma: X \rightarrow E/G$ is trivial, we define
	\begin{align*}
	\deg^{\rm an} E(h, \chi):=\frac{\sqrt{-1}}{2\pi} \int_{X_{\boldsymbol{D}}} \chi_* F_{h}.
	\end{align*}
	A metrized $G$-bundle  $(E,\partial''_E,h)$ is of \emph{degree zero} if we have $\deg^{\rm an} E(h, \chi)=0$ for any character $\chi: P \rightarrow \mathbb{C}^*$.
\end{defn}

Now we move to Higgs bundles and connections.
\begin{defn}\label{defn_met_Higgs_and_conn}
	A \emph{metrized $G$-Higgs bundle} on $X_{\boldsymbol D}$ is a tuple $(E,\partial''_E, \phi,h)$, where
	\begin{itemize}
		\item $(E,\partial''_E,h)$ is a metrized $G$-bundle on $X_{\boldsymbol D}$,
		\item $\phi$ is a Higgs field, that is, a holomorphic section of $E(\mathfrak{g}) \otimes K_{X_{\boldsymbol D}}$, where $E(\mathfrak{g})$ is the adjoint bundle.
	\end{itemize}
	A \emph{metrized $G$-connection} on $X_{\boldsymbol D}$ is a triple $(E,D,h)$, where
	\begin{itemize}
		\item $E$ is a $G$-bundle on $X_{\boldsymbol D}$ equipped with an adapted metric $h$,
		\item $D$ is an integrable connection on $E$.
	\end{itemize}
\end{defn}

Given a metrized $G$-Higgs bundle $(E,\partial''_E,\phi,h)$, we obtain a connection 
\begin{equation}\label{induced_flat_connect}
D=\partial'_h +\partial''_E + \phi + \phi^*,
\end{equation}
where $\partial'_h+\partial''_E$ is the Chern connection which preserves the metric $h$ and $\phi^*$ is adjoint to $\phi$ with respect to the metric $h$. Denote by $F_D:= D^2$ the curvature of $D$. Furthermore, we define 
\begin{align*}
D'':=\partial''_E + \phi,
\end{align*}
and denote by
\begin{align*}
F_{D''}:=(D'')^2
\end{align*}
the \emph{pseudo-curvature} of $D''$. Clearly, $F_{D''}=0$. In conclusion, we obtain a tuple $(E,D,h)$ from a metrized $G$-Higgs bundle. If the curvature $F_D = 0$, then $(E,D,h)$ is a metrized $G$-connection.

For the other direction, given a metrized $G$-connection $(E,D,h)$, we write $D=d'+d''$ as a sum of operators of type $(1,0)$ and $(0,1)$. In this way, $d''$ gives a new holomorphic structure on $E$, and denote by $V:=(E,d'')$ the corresponding holomorphic $G$-bundle. The $(1,0)$-part $d'$ gives a connection on the holomorphic $G$-bundle $V$. Then, the $G$-connection $(E,D)$ can be regarded as $(V,d')$, which is an algebraic (holomorphic) description. Thus, a metrized $G$-connection $(E,D,h)$ is also regarded as a triple $(V,d',h)$. With respect to the given metric $h$, we can find a $(1,0)$-operator $\delta'$ and a $(0,1)$-operator $\delta''$ such that both $d'+\delta''$ and $\delta'+d''$ preserve the metric $h$. Then, we define
\begin{equation}\label{eq_metric_Higgs}
\partial'':=\frac{1}{2}(d''+\delta''), \quad \phi:=\frac{1}{2}(d'-\delta').
\end{equation}
Therefore, we obtain a tuple $(E,\partial'',\phi,h)$. If the pseudo-curvature $F_{D''}=0$, where $D'':=\partial''+\phi$, then the tuple $(E,\partial'',\phi,h)$ is a metrized $G$-Higgs bundle. Moreover, we have the following relation with respect to the above data
\begin{equation}
F_D = F_h + 2 F_{D''} + [\phi, \phi^*].
\end{equation}
This implies that if $D$ is an integrable connection and the pseudo-curvature $F_{D''}$ vanishes, then the metrized $G$-bundle $(E,\partial'',h)$ is of degree zero. 

Let $\boldsymbol\theta$ be a collection of weights and $\boldsymbol{n}$ a collection of positive integers indexed by points in $\boldsymbol D$. We next introduce a special case of $G$-Higgs bundles and $G$-connections.
\begin{defn}\label{defn_ana_adapted}
	A metrized $G$-Higgs bundle $(E,\partial''_E,\phi,h)$ is called \emph{$(\boldsymbol\theta,\boldsymbol{n})$-adapted} if for each puncture $x \in \boldsymbol D$, we have
	\begin{align*}
	z^{\theta_x} \cdot (z^{n_x} \phi(z)) \cdot z^{-\theta_x} \text{ is bounded as $z$ approaches zero},
	\end{align*}
	where we take a local trivialization on a neighborhood of $x$ such that $\partial''_E=\bar{\partial}_z$ is the $(0,1)$-part of the standard connection with $z$ the local coordinate vanishing at $x$. If the above condition for Higgs fields is satisfied for $n_x=1$, then the Higgs field is called \emph{tame}, otherwise, it is called \emph{wild}.  
	
	A metrized $G$-connection $(V,d',h)$ (or $(E,D,h)$) is called \emph{$(\boldsymbol\theta,\boldsymbol{n})$-adapted} if for each puncture $x \in \boldsymbol{D}$, we have
	\begin{align*}
	z^{\theta_x} \cdot (z^{n_x} d') \cdot z^{-\theta_x} \text{ is bounded as $z$ approaches zero},
	\end{align*}
	where we take a local trivialization on a neighborhood of $x$ such that $\partial''_E=\bar{\partial}_z$ is the $(0,1)$-part of the standard connection with $z$ the local coordinate vanishing at $x$. If the above condition for connections is satisfied for $n_x=1$, then it is called \emph{tame}, otherwise, it is called \emph{wild}.
	
	Moreover, a holomorphic reduction of structure group $\sigma: X_{\boldsymbol D} \rightarrow E/P$ is called \emph{$\boldsymbol\theta$-adapted}, if we have
	\begin{align*}
	z^{\theta_x} \cdot \sigma(z) \cdot z^{-\theta_x} \text{ is bounded as $z$ approaches zero}
	\end{align*}
	for each puncture $x \in \boldsymbol{D}$. 
\end{defn}

Similar to merohoric Higgs torsors and connections, we can define \emph{compatible $\boldsymbol\theta$-adapted holomorphic reductions of structure group} for metrized Higgs bundles and connections (see \cite[\S 4]{HKSZ}). Then, the stability condition for the analytic side is given as follows:
\begin{defn}\label{defn_ana_stab}
	A metrized $(\boldsymbol\theta,\boldsymbol{n})$-adapted $G$-Higgs bundle $(E,\partial''_E, \phi,h)$ is \emph{$R_h$-stable} (resp. \emph{$R_h$-semistable}), if for
	\begin{itemize}
		\item any proper parabolic group $P \subseteq G$,
		\item any compatible $\boldsymbol\theta$-adapted holomorphic reduction of structure group $\sigma: X_{\boldsymbol D} \rightarrow E/P$,
		\item any nontrivial anti-dominant character $\chi: P \rightarrow \mathbb{C}^*$, which is trivial on the center of $P$,
	\end{itemize}
	one has
	\begin{align*}
	\deg^{\rm an} E (h, \sigma, \chi) > 0, \quad (\text{resp. } \geq 0).
	\end{align*}
	
	A metrized $(\boldsymbol\theta,\boldsymbol{n})$-adapted $G$-connection $(V,d',h)$ is \emph{$R_h$-stable} (resp. \emph{$R_h$-semistable}), if for
	\begin{itemize}
		\item any proper parabolic group $P \subseteq G$,
		\item any compatible $\boldsymbol\theta$-adapted holomorphic reduction of structure group $\sigma: X_{\boldsymbol D} \rightarrow V/P$,
		\item any nontrivial anti-dominant character $\chi: P \rightarrow \mathbb{C}^*$, which is trivial on the center of $P$,
	\end{itemize}
	one has
	\begin{align*}
	\deg^{\rm an} V (h, \sigma,\chi) > 0, \quad (\text{resp. } \geq 0).
	\end{align*}
\end{defn}

A notion of polystability condition can be given for $G$-bundles. Since the polystability condition only appears in the statement of a version of Kobayashi--Hitchin correspondence (Theorem \ref{thm_KH_corr}) in this paper, we refer the reader to \cite[\S 3.8]{HKSZ} for more details. 	

\subsection{Correspondence}\label{subsect_func_Xi}

We fix some notations for categories:
\begin{itemize}
    \item $\mathcal{C}(X_{\boldsymbol D},G,\boldsymbol\theta)$: the category of metrized $G$-bundles $(E,\partial''_E,h)$ with $\boldsymbol\theta$-adapted metrics $h$ on $X_{\boldsymbol{D}}$;
    \item $\mathcal{C}(X,\mathcal{G}_{\boldsymbol\theta})$: the category of parahoric $\mathcal{G}_{\boldsymbol\theta}$-torsors $\mathcal{E}$ on $X$;
	\item $\mathcal{C}_{\rm Higgs}(X,\mathcal{G}_{\boldsymbol\theta},\boldsymbol{n})$: the category of merohoric $\mathcal{G}_{\boldsymbol\theta}$-Higgs torsors $(\mathcal{E},\varphi)$ on $X$, where $\varphi: X \rightarrow \mathcal{E}(\mathfrak{g}) \otimes K_X(\boldsymbol{nD})$ is a meromorphic Higgs field;
	\item $\mathcal{C}_{\rm Conn}(X,\mathcal{G}_{\boldsymbol\theta},\boldsymbol{n})$: the category of merohoric $\mathcal{G}_{\boldsymbol\theta}$-connections $(\mathcal{V},\nabla)$ on $X$, where $\nabla: \mathcal{O}_{\mathcal{V}} \rightarrow \mathcal{O}_{\mathcal{V}} \otimes K_X(\boldsymbol{nD})$ is a meromorphic connection;
	\item $\mathcal{C}_{\rm Higgs}(X_{\boldsymbol{D}},G,\boldsymbol\theta,\boldsymbol{n})$: the category of metrized  $(\boldsymbol\theta,\boldsymbol{n})$-adapted $G$-Higgs bundles on $X_{\boldsymbol{D}}$;
	\item $\mathcal{C}_{\rm Conn}(X_{\boldsymbol{D}},G,\boldsymbol\theta,\boldsymbol{n})$: the category of metrized  $(\boldsymbol\theta,\boldsymbol{n})$-adapted $G$-connections on $X_{\boldsymbol{D}}$;
\end{itemize}

In this subsection, we define two functors relating analytic and algebraic objects:
\begin{align*}
\Xi_{\rm Higgs} & : \mathcal{C}_{\rm Higgs}(X_{\boldsymbol{D}},G,\boldsymbol\theta,\boldsymbol{n}) \rightarrow \mathcal{C}_{\rm Higgs}(X,\mathcal{G}_{\boldsymbol\theta},\boldsymbol{n}), \\
\Xi_{\rm Conn} & : \mathcal{C}_{\rm Conn}(X_{\boldsymbol{D}},G,\boldsymbol\theta,\boldsymbol{n}) \rightarrow \mathcal{C}_{\rm Conn}(X,\mathcal{G}_{\boldsymbol\theta},\boldsymbol{n})  . 
\end{align*}
We first review the construction of the functor
\begin{align*}
    \Xi: \mathcal{C}(X_{\boldsymbol D},G,\boldsymbol\theta) \rightarrow \mathcal{C}(X,\mathcal{G}_{\boldsymbol\theta})
\end{align*}
given in \cite[\S 3.1]{HKSZ}. Given a metrized $G$-bundle $(E,\partial''_E,h)$ with $h$ a $\boldsymbol{\theta}$-adapted metrix on $X_{\boldsymbol{D}}$, we want to extend it to a parahoric $\mathcal{G}_{\boldsymbol\theta}$-torsor on $X$. Since such parahoric $\mathcal{G}_{\boldsymbol\theta}$-torsors on $X$ are classified by
\begin{align*}
	\left[ \prod_{x \in \boldsymbol D} G_{\theta_x}(K) \backslash \prod_{x \in \boldsymbol D} G(K) / G(\mathbb{C}[X_{\boldsymbol D}])  \right],
\end{align*}
the extension is not unique in general. We pick the identity element in the double coset and denote by $\mathcal{E}$ the corresponding parahoric $\mathcal{G}_{\boldsymbol\theta}$-torsor, of which the restriction $\mathcal{E}|_{X_{\boldsymbol{D}}}$ is isomorphic to the holomorphic $G$-bundle $E$. With respect to this particular choice of extension, we get the functor $\Xi$
\begin{align*}
    \Xi: \mathcal{C}(X_{\boldsymbol D},G,\boldsymbol\theta) \rightarrow \mathcal{C}(X,\mathcal{G}_{\boldsymbol\theta}).
\end{align*}
When $G={\rm GL}_n(\mathbb{C})$, this construction is exactly the same as in \cite[\S 10]{Simp1988} and \cite[\S 3]{Simp}. 

Now we generalize the functor $\Xi$ to Higgs bundles and connections. We take meromorphic Higgs bundles as an example. Let $(E,\partial''_E,\phi,h)$ be a metrized $(\boldsymbol\theta,\boldsymbol{n})$-adapted $G$-Higgs bundle on $X_{\boldsymbol{D}}$, and denote by $\mathcal{E}$ the corresponding parahoric $\mathcal{G}_{\boldsymbol\theta}$-torsor of $(E,\partial''_E,h)$ under the functor $\Xi$. Note that the meromorphic Higgs field $\phi$ satisfies the condition 
\begin{align*}
	z^{\theta_x} \cdot (z^{n_x} \phi(z)) \cdot z^{-\theta_x} \text{ is bounded as $z$ approaches zero},
\end{align*}
for each puncture $x \in \boldsymbol{D}$, which implies that $z^{\theta_x} \cdot (z^{n_x} \phi(z)) \cdot z^{-\theta_x}$ is in $\mathfrak{g}(R)dz$. With the help of Remark \ref{rem_mero_ana_to_alg}, the meromorphic Higgs field $\phi(z)$ is a well-defined element in $\mathfrak{g}_{\theta_x}(K) dz/ z^{n_x}$. Thus, the Higgs field $\phi$ can be extended to a meromorphic Higgs field $\varphi \in H^0(X,\mathcal{E}(\mathfrak{g}) \otimes K_X(\boldsymbol{nD}))$ naturally. In sum, given a metrized $(\boldsymbol\theta,\boldsymbol{n})$-adapted $G$-Higgs bundle $(E,\partial''_E,\phi,h)$, we construct a merohoric $\mathcal{G}_{\boldsymbol\theta}$-Higgs torsor $(\mathcal{E},\varphi)$ on $X$, and then, we obtain a functor and denote it by
\begin{align*}
    \Xi_{\rm Higgs} & : \mathcal{C}_{\rm Higgs}(X_{\boldsymbol{D}},G,\boldsymbol\theta,\boldsymbol{n}) \rightarrow \mathcal{C}_{\rm Higgs}(X,\mathcal{G}_{\boldsymbol\theta},\boldsymbol{n}).
\end{align*}
The functor for meromorphic connections
\begin{align*}
    \Xi_{\rm Conn}: \mathcal{C}_{\rm Conn}(X_{\boldsymbol{D}},G,\boldsymbol\theta,\boldsymbol{n}) \rightarrow \mathcal{C}_{\rm Conn}(X,\mathcal{G}_{\boldsymbol\theta},\boldsymbol{n}).
\end{align*}
can be defined in a similar way. 

These two functors $\Xi_{\rm Higgs}$ and $\Xi_{\rm Conn}$ not only build a bridge between analytic and algebraic objects, but also preserve the degree and stability conditions. 
\begin{prop}\label{prop_ana=alg Higgs and conn}
	A metrized $(\boldsymbol\theta,\boldsymbol{n})$-adapted $G$-Higgs bundle is $R_h$-stable (resp. $R_h$-semistable) if and only if the corresponding merohoric $\mathcal{G}_{\boldsymbol\theta}$-Higgs torsor is $R$-stable (resp. $R$-semistable). Similarly, a metrized $(\boldsymbol\theta,\boldsymbol{n})$-adapted $G$-connection is $R_h$-stable (resp. $R_h$-semistable) if and only if the corresponding merohoric $\mathcal{G}_{\boldsymbol\theta}$-connection is $R$-stable (resp. $R$-semistable).
\end{prop}

\begin{proof}
We only give the proof for Higgs bundles. Given an object $(E,\partial''_E,\phi,h) \in \mathcal{C}_{\rm Higgs}(X_{\boldsymbol{D}},G,\boldsymbol\theta,\boldsymbol{n})$, let $(\mathcal{E},\varphi)$ be the corresponding merohoric $\mathcal{G}_{\boldsymbol\theta}$-Higgs torsor on $X$. There is a one-to-one correspondence between
\begin{itemize}
	\item characters $\chi$ of $P$ and characters $\kappa$ of $\mathcal{P}_{\boldsymbol\theta}$ (\cite[Lemma 4.2]{KSZ2parh});
	\item compatible $\boldsymbol\theta$-adapted holomorphic reductions $\sigma: X_{\boldsymbol D} \rightarrow E/P$ and compatible reductions $\varsigma: X \rightarrow \mathcal{E}/\mathcal{P}_{\boldsymbol\theta}$ (\cite[Lemma 3.14 and Lemma 5.6]{HKSZ}).
\end{itemize}
Under the above one-to-one correspondence, we have
\begin{align*}
\deg^{\rm an} E(h, \sigma, \chi)=parh\deg \mathcal{E}(\varsigma,\kappa),
\end{align*}
and in other words, the analytic degree equals to the algebraic (parahoric) degree (\cite[Proposition 3.15]{HKSZ}). With respect to the above discussion, the proposition follows directly from Definition \ref{defn_alg_stab} and Definition \ref{defn_ana_stab}.
\end{proof}

At the end of this section, we define subcategories $\mathcal{C}_{\rm Dol} \subseteq \mathcal{C}_{\rm Higgs}$ (resp. $\mathcal{C}_{\rm dR} \subseteq \mathcal{C}_{\rm Conn}$) for Higgs bundles (resp. integrable connections) by adding two additional conditions: $R$-stable and degree zero. Then, we obtain two well-defined functors by Proposition \ref{prop_ana=alg Higgs and conn}:
\begin{align*}
\Xi_{\rm Dol} & : \mathcal{C}_{\rm Dol}(X_{\boldsymbol{D}},G,\boldsymbol\theta,\boldsymbol{n}) \rightarrow \mathcal{C}_{\rm Dol}(X,\mathcal{G}_{\boldsymbol\theta},\boldsymbol{n}), \\
\Xi_{\rm dR} & : \mathcal{C}_{\rm dR}(X_{\boldsymbol{D}},G,\boldsymbol\theta,\boldsymbol{n}) \rightarrow \mathcal{C}_{\rm dR}(X,\mathcal{G}_{\boldsymbol\theta},\boldsymbol{n}). 
\end{align*}

\section{Dolbeault Category vs. de Rham Category}\label{sect_Dol_dR}

In this section, we introduce the definition of \emph{irregular type} for meromorphic connections and meromorphic Higgs fields (Definition \ref{defn_irregular_type_parh}). With respect to the additional data of irregular types, we construct a model metric for a meromorphic connection and then give a version of Kobayashi--Hitchin correspondence (Theorem \ref{thm_KH_corr}) following the proof in \cite[\S 6]{HKSZ}. In this way, we obtain a correspondence between $R$-stable merohoric Higgs torsors of degree zero and $R$-stable merohoric connections of degree zero (Theorem \ref{thm_alg_dR_Dol}).

\subsection{Canonical Forms and Irregular Types}\label{subsect_cano_form}
In this subsection, we sometimes add subscripts $R_z:=\mathbb{C}[\![z]\!]$ and $K_z:=\mathbb{C}(\!(z)\!)$ to emphasize the choice of local coordinate. With the same notation as in \S\ref{sect_Pre}, an element $B(z) \in \mathfrak{g}(K)$ can be written in the following form
\begin{align*}
B(z) = \sum_{i \geq -n} B_i z^i = B_{-n} z^{-n} + \dots + B_{-1} z^{-1} + \cdots ,
\end{align*}
where $B_{-n} \neq 0$. Denote by $d+ B(z)\frac{dz}{z}$ the corresponding \emph{$G$-connection}, where $d$ is the exterior differential operator. In this paper, we focus on \emph{meromorphic $G$-connections}, i.e. $n \geq 0$. Given $g \in G(K)$, there is a natural action of $g$ on $d + B(z)\frac{dz}{z}$, i.e.
\begin{align*}
g \circ (d + B(z) \frac{dz}{z}) = d + \left( - dg \cdot g^{-1} + {\rm Ad}(g) \cdot B(z) \frac{dz}{z} \right),
\end{align*}
which is called the \emph{gauge action}. Two $G$-connections $d+A(z)\frac{dz}{z}$ and $d + B(z)\frac{dz}{z}$ are \emph{gauge equivalent} if there exists $g \in G(K)$ such that $g \circ (d+A(z)\frac{dz}{z}) = d+ B(z)\frac{dz}{z}$.

\begin{defn}[Canonical form]\label{defn_cano_form}
	A meromorphic connection $d+B(z)\frac{dz}{z}$, where
	\begin{align*}
	B(z) = B_{-n} z^{-n} + \cdots + B_{0} + \cdots ,
	\end{align*}
	is said to be in \emph{canonical form} if
	\begin{itemize}
		\item the holomorphic part of $d+ B(z)\frac{dz}{z}$ vanishes, i.e. $B(z)=B_{-n} z^{-n} + \dots + B_{0}$,
		\item $B_i$ is a semisimple element in $\mathfrak{g}$ for $i \leq -1$,
		\item the elements $B_{-n},\dots,B_0$ are pairwise commutative.
	\end{itemize}
\end{defn}

Note that if $B(z)$ is in canonical form, the coefficients $B_i$ for $i \leq -1$ can be supposed to be in $\mathfrak{t}$ under conjugation. The following important result says that given any meromorphic $G$-connection, there exists a cover, expressed in local coordinates as $z=w^d$ for a positive integer $d$, such that the connection (in variable $w$) is gauge equivalent to one in canonical form:
\begin{thm}[Theorem in {\cite[\S9.5]{BaVa}} or  Theorem 4.3 in \cite{Herr}]\label{thm_canonical_form}
	Let $G$ be a connected complex reductive group. Let $d + B(z)\frac{dz}{z}$ be a meromorphic $G$-connection. There exists a positive integer $d$ and an element $g \in G(K_w)$, where $z=w^d$, such that $g \circ (d+ B(w) \frac{d(w^d)}{w^d} )$ is in canonical form.
\end{thm}
This result for $G={\rm GL}_n(\mathbb{C})$ is well-known especially in the area of ordinary differential equations (see \cite{Malg,Wasow,vanSin} for instance).  

\begin{rem}\label{rem_cano_form_proof}
	We would like to briefly state the idea of the proof of the theorem. Given a meromorphic $G$-connection $d + B(z)\frac{dz}{z}$, we suppose that the meromorphic part $B(z)$ is not nilpotent, otherwise we can take an appropriate element $g=z^{\theta}$ such that $g \circ (d + B(z)\frac{dz}{z})$ is a logarithmic connection, where $\theta$ is a co-character. 
	
	Under this assumption, although the meromorphic part of $B(z)$ is not nilpotent, the leading coefficient $B_{-n}$ might be nilpotent. Thus, the first step is to find an appropriate element $g_1$ such that the first nontrivial coefficient of 
	\begin{align*}
	d+ C(z)\frac{dz}{z}=g_1 \circ (d+ B(z)\frac{dz}{z}) 
	\end{align*}
	is semisimple. In this step, $g_1$ is usually in the form $z^{\theta_1}$, where $\theta_1$ is a rational weight. This actually implies that the element $g$ we find in the theorem is in $G(K_w)$ with $z=w^d$ for some positive integer $d$. 
	
	Next, the leading coefficient $C_{-n_c}$ of $C(z)$ is not regular semisimple in general. Despite this, we can find an appropriate element $g_2$ and obtain a new connection 
	\begin{align*}
	d+ D(z)\frac{dz}{z}=g_2 \circ (d+ C(z)\frac{dz}{z}) 
	\end{align*}
	such that if a nilpotent element in $\mathfrak{g}$ commutes with the leading coefficient $D_{-n_d}$ of $D(z)$, then it could only be (part of) the coefficient of $z^{i}$ in $D(z)$ with $i \geq 0$. 
	
	For the last step, we can choose $g_3 \in G(K_z)$ such that $g_3 \circ (d + D(z)\frac{dz}{z})$ is in canonical form. In this step, $g_3$ is a product of elements in the form $\exp(X_k z^k)$, where $X_k \in \mathfrak{g}$. Moreover, for the last step, $g_3$ is actually in $G(R)$.
\end{rem}

Based on the result of Theorem \ref{thm_canonical_form}, we introduce the definition of \emph{irregular types} for meromorphic $G$-connections.

\begin{defn}\label{defn_irregular_type}
For a meromorphic $G$-connection and a choice of local coordinate $z$ at the points in $\boldsymbol{D}$, an \emph{irregular type} is an element $Q(z) \in \mathfrak{t}(K)/\mathfrak{t}(R)$ (or $Q$ for simplicity), and such an element is usually regarded as a $\mathfrak{t}$-valued meromorphic function (modulo holomorphic terms). 
	
Let $d + B(z)\frac{dz}{z}$ be a meromorphic $G$-connection. Given an element $Q \in \mathfrak{t}(K)/\mathfrak{t}(R)$, if $d + B(z)\frac{dz}{z}$ is equivalent to a connection in the form
\begin{align*}
d + dQ + (\text{terms with degree $\geq 0$})\frac{dz}{z}
\end{align*}
under the (gauge) action of $G(R)$, then we say that the connection $d + B(z)\frac{dz}{z}$ is with \emph{irregular type $Q$}.
	
Let $A(z) \in \mathfrak{g}(K)$ be an element and take $A(z)\frac{dz}{z}$ as a meromorphic $G$-Higgs field. If  $A(z)\frac{dz}{z}$ is equivalent to a Higgs field in the form
\begin{align*}
dQ + (\text{terms with degree $\geq 0$ }) \frac{dz}{z}
\end{align*} 
under the (adjoint) action of $G(R)$, then we say that the Higgs field $A(z)\frac{dz}{z}$ is with \emph{irregular type $Q$}. 
\end{defn}

In this paper, we always assume that the irregular type is nontrivial because when it is trivial, the study of the regular singular connections goes back to the tame case \cite{Bo}. Now we move to the parahoric side and give the definition of \emph{irregular types} for parahoric objects considered in this paper. 

\begin{defn}\label{defn_irregular_type_parh}
	Given an irregular type $Q \in \mathfrak{t}(K)/\mathfrak{t}(R)$ and an element $B(z) \in \mathfrak{g}_\theta(K)/z^n$, the meromorphic $G_\theta(K)$-connection $d+B(z)\frac{dz}{z}$ is with \emph{irregular type $Q$}, if it is equivalent to a connection in the form
	\begin{align*}
	d + dQ + (\sum_{i=0}^{\infty} B'_i z^i) \frac{dz}{z}
	\end{align*}
	under the (gauge) action of $G_\theta(K)$, where $\sum_{i=0}^{\infty} B'_i z^i \in \mathfrak{g}_\theta(K)$. Similarly, a $G_\theta(K)$-Higgs field $A(z)\frac{dz}{z}$ is with \emph{irregular type $Q$} if it is equivalent to one in the form
	\begin{align*}
	    dQ + (\sum_{i=0}^{\infty} A'_i z^i) \frac{dz}{z}
	\end{align*}
	under the (adjoint) action of $G_\theta(K)$, where $\sum_{i=0}^{\infty} A'_i z^i \in \mathfrak{g}_\theta(K)$.
\end{defn}

\begin{rem}
In this definition, given an irregular type $Q$, a meromorphic $G_\theta(K)$-connection is of irregular type $Q$ if it is gauge equivalent (under $G_\theta(K)$) to one in the desired form $d + dQ + (\text{terms with degree} \geq 0) \frac{dz}{z}$. However, there exist meromorphic $G_\theta(K)$-connections which cannot be transformed into the form $d + dQ + (\text{terms with degree} \geq 0) \frac{dz}{z}$ for any irregular type $Q$. The statement is also true for meromorphic $G$-connections, and moreover, this is one of the reasons why one has to find an appropriate cover to transform meromorphic connections to the ones in canonical form. 
\end{rem}

\begin{prop}\label{prop_canonical_form_parah}
    Given a weight $\theta$ and an irregular type $Q$, let $d + B(z)\frac{dz}{z}$ be a meromorphic $G_\theta(K)$-connection with irregular type $Q$. There exists an element $g \in G_\theta(K)$ such that the meromorphic $G_\theta(K)$-connection
    \begin{align*}
        d + A(z)\frac{dz}{z} : = g \circ (d + B(z)\frac{dz}{z})
    \end{align*}
    is in canonical form. Moreover, $A_0 \in \mathfrak{p}_\theta$ (see Notation \ref{notn_parab_levi}).
\end{prop}

\begin{proof}
    We extend the argument of \cite[Theorem 6]{Bo} and \cite[Theorem 4.3]{Herr} to the merohoric case. We first briefly review the graded structure of $\mathfrak{g}$ and $\mathfrak{g}_\theta (K)$ given by $\theta$. For each real number $\lambda$, denote by $\mathfrak{g}_\lambda \subseteq \mathfrak{g}$ the $\lambda$-eigenspace of the adjoint operator ${\rm ad}(\theta)$. Clearly, we have
    \begin{align*}
        \mathfrak{g} = \bigoplus_{\lambda \in \mathbb{R}} \mathfrak{g}_{\lambda} \quad \text{ and } \quad [\mathfrak{g}_{\lambda_1}, \mathfrak{g}_{\lambda_2}] \subseteq \mathfrak{g}_{\lambda_1 + \lambda_2}.
    \end{align*}
    Since $\mathfrak{g}$ is a finite dimensional vector space, there are only finitely many nontrivial subspaces $\mathfrak{g}_\lambda$. In a similar way, we define
    \begin{align*}
        \mathfrak{g}(K)_\mu := \{ \sum X_i z^i \in \mathfrak{g}(K) \text{ } | \text{ } X_i \in \mathfrak{g}_\lambda, \text{ where $\lambda+i=\mu$} \}
    \end{align*}
    for all $\mu \in \mathbb{R}$, which is a finite dimensional vector space. We also have
    \begin{align*}
        [\mathfrak{g}(K)_{\mu_1},\mathfrak{g}(K)_{\mu_2}] \subseteq \mathfrak{g}(K)_{\mu_1 + \mu_2}\quad \text{ and } \quad \prod_{\mu \in \mathbb{R}_{\geq 0}} \mathfrak{g}(K)_{\mu} = \mathfrak{g}_\theta(K).
    \end{align*}
    Let $0 = \mu_0 < \mu_1 < \cdots$ be the sequence of real numbers such that $\mathfrak{g}(K)_{\mu_n}$ is nontrivial. We also introduce the notation $\mathfrak{g}(K)_{ > \mu'} := \prod_{\mu > \mu'} \mathfrak{g}(K)_{\mu}$, and then $\mathfrak{g}_\theta(K)=\mathfrak{g}(K)_{\geq 0}$.
    
    Now we go back to the proof of the proposition and suppose that the meromorphic connection $d + B(z) \frac{dz}{z}$ is exactly in the form
    \begin{align*}
        B(z)\frac{dz}{z} = dQ + B_+(z) \frac{dz}{z},
    \end{align*}
    where $B_{i} \in \mathfrak{t}$ for $i \leq -1$ and $B_+(z):=\sum_{i=0}^{\infty} B_i z^i \in \mathfrak{g}_\theta(K)$. Denote by $B_{\mu_i}(z)$ the corresponding piece of $B_+(z)$ in $\mathfrak{g}(K)_{\mu_i}$. Thus, $B_+(z) = \sum_{i \geq 0} B_{\mu_i}(z)$ and each summand $B_{\mu_i}(z) = \sum B_{\mu_i j} z^j$ is a finite sum. It is enough to show that each piece $B_{\mu_i}(z)$ of $B_+(z)$ can be transformed to one commuting with $B_{-i}$ for all $i \leq -1$. If the piece $B_{\mu_i}(z)$ satisfies the condition, we say that the $\mu_i$-piece $B_{\mu_i}(z)$ is normalized.
    
    We start from the base step $i=0$. We claim that there exists an element $X_{\mu_0}(z) \in \mathfrak{g}(K)_{\mu_0}$ such that $B_{\mu_0}(z) + [X_{\mu_0}(z),B_{-n}]$ lies in the zero eigenspace of ${\rm ad}(B_{-n})$. This claim follows from the proof of \cite[Theorem 6]{Bo}. Then, the gauge action of
    ${\rm exp}(X_{\mu_0}(z)z^n)$ on $d + B(z)\frac{dz}{z}$ satisfies the following properties
    \begin{itemize}
        \item it does not change the part of $B(z)$ with negative degree;
        \item it holds that
        \begin{align*}
            [X_{\mu_0}(z)z^n, B_{\mu_i}(z)] \in \mathfrak{g}(K)_{\mu_0+\mu_i+n};
        \end{align*}
        \item it holds that 
        \begin{align*}
            \frac{\partial {\rm exp}(X_{\mu_0}(z)z^n)}{\partial z} ({\rm exp}(X_{\mu_0}(z)z^n))^{-1} \in \mathfrak{g}(K)_{\mu_0+n}.
        \end{align*}
    \end{itemize}
    We then have
    \begin{align*}
        {\rm exp}(X_{\mu_0}(z)z^n) \circ (d + B(z)\frac{dz}{z}) = d + dQ + (B_{\mu_0}(z) + [X_{\mu_0}(z),B_{-n}] )\frac{dz}{z} \text{ mod } \mathfrak{g}(K)_{> \mu_0} \frac{dz}{z}.
    \end{align*}
    Thus, the $\mu_0$-piece $B_{\mu_0}(z) + [X_{\mu_0}(z),B_{-n}]$ of the new meromorphic connection ${\rm exp}(X_{\mu_0}(z)z^n) \circ (d + B(z)\frac{dz}{z})$ commutes with $B_{-n}$. For simplicity, we use the same notation $B(z)\frac{dz}{z}$ for the resulting meromorphic connection, of which the $\mu_0$-piece $B_{\mu_0}(z)$ commutes with $B_{-n}$. With the same process, we find an element $X'_{\mu_0}(z) \in \mathfrak{g}_{\mu_0}$ such that $B_{\mu_0}(z) + [X'_{\mu_0}(z),B_{-{n-1}}]$ lies in the zero eigenspace of ${\rm ad}(B_{-(n-1)})$. Since $B_{\mu_0}(z)$ commutes with $B_{-n}$, the element $X'_{\mu_0}(z)$ also commutes with $B_{-n}$, and then 
    \begin{align*}
        {\rm Ad}({\rm exp}(X'_{\mu_0}(z)z^{n-1})) B_{-n} = B_{-n}.
    \end{align*}
    Thus, under the gauge action of ${\rm exp}(X'_{\mu_0}(z)z^{n-1})$, we obtain a meromorphic connection, of which the $\mu_0$-piece commutes with both $B_{-n}$ and $B_{-(n-1)}$. In sum, repeating this process, we can find an element $X_{\mu_0}(z)$ such that the $\mu_0$-piece of the meromorphic connection ${\rm Ad}({\rm exp}(X_{\mu_0}(z))) \circ (d + B(z)\frac{dz}{z})$ is normalized.
    
    For the inductive step, suppose that the $\mu_i$-piece $B_{\mu_i}(z)$ of the meromorphic connection $d+ B(z)\frac{dz}{z}$ has been normalized for $0 \leq i \leq k-1$. We can choose $X_{\mu_k}(z) \in \mathfrak{g}(K)_{\mu_k}$ such that $B_{\mu_k}(z) + [X_{\mu_k}(z),B_{-n}]$ lies in the zero eigenspace of ${\rm ad}(B_{-n})$. Now, we consider the element ${\rm exp}(X_{\mu_k}(z)z^n)$. By the graded structure of $\mathfrak{g}_\theta(K)$, the operator ${\rm exp}(X_{\mu_k}(z)z^n)$ fixes the $\mu_i$-piece $B_{\mu_i(K)}$. Therefore, applying the induction hypothesis, we have
    \begin{align*}
        {\rm exp}(X_{\mu_k}(z)z^n) \circ (d + B(z)\frac{dz}{z}) & = d + dQ + \left( \sum_{i=0}^{k-1} B_{\mu_i}(z) \right) \\
        & + (B_{\mu_k}(z) + [X_{\mu_k}(z),B_{-n}] )\frac{dz}{z} \text{ mod } \mathfrak{g}(K)_{> \mu_k} \frac{dz}{z}.
    \end{align*}
    With the same process as we did in the base step, the $\mu_k$-piece $B_{\mu_k}(z)$ can be normalized. 
\end{proof}

\begin{rem}\label{rem_diff_tame_wild}
    Note that in the setup of this paper (Definition \ref{defn_irregular_type_parh}), Proposition \ref{prop_canonical_form_parah} shows that a meromorphic connection with irregular type $Q$ is gauge equivalent to one in canonical form without working on an appropriate cover. When the irregular type $Q$ is trivial, the result of Proposition \ref{prop_canonical_form_parah} is not true in general, because the coefficients in the holomorphic part cannot be always transformed to elements commuting with the leading coefficient, which is $B_0$ in this case (see \cite[Theorem 6]{Bo}).
\end{rem}

\begin{cor}\label{cor_canonical_form}
	There is a one-to-one correspondence between the $G_\theta(K)$-orbits of meromorphic $G_\theta(K)$-connections $d + A(z)\frac{dz}{z}$ with irregular type $Q$, where $A(z) \in \mathfrak{g}_\theta(K)/z^n$, and the $G(R)$-orbits of meromorphic $G$-connections $d + B(z)\frac{dz}{z}$ with irregular type $Q$, where $B(z) \in \mathfrak{g}(R)/z^n$ with $B_0 \in \mathfrak{p}_\theta$.
\end{cor}

\begin{proof}
    By Theorem \ref{thm_canonical_form} and Proposition \ref{prop_canonical_form_parah}, we know that both meromorphic $G$-connections and meromorphic $G_\theta(K)$-connections with irregular type $Q$ can be transformed into a unique canonical form under the (gauge) action of $G(R)$ and $G_\theta(K)$ respectively. Now, between the canonical forms, there is an obvious one-to-one correspondence. Thus, we prove the lemma.
\end{proof}

\subsection{Statement of Results}
The setup in this subsection is the same as in \S\ref{sect_Pre}. Let $X$ be a smooth projective curve over $\mathbb{C}$. We fix a reduced effective divisor $\boldsymbol{D}$ on $X$, which is also regarded as a set of distinct points. Denote by $X_{\boldsymbol D}:= X\backslash \boldsymbol{D}$ the punctured curve. We also fix a connected complex reductive group $G$. Let $\boldsymbol{Q}=\{Q_x, x \in \boldsymbol{D}\}$ be a set of irregular types indexed by points in $\boldsymbol D$. Denote by $\boldsymbol{n}=\{n_x, x \in \boldsymbol{D}\}$ the set of positive integers, where $n_x = \deg Q_x + 1$.
\begin{defn}
	A merohoric $\mathcal{G}_{\boldsymbol\theta}$-Higgs torsor $(\mathcal{E},\varphi)$ is with \emph{irregular type} $\boldsymbol{Q}$, if for each puncture $x \in \boldsymbol{D}$, the meromorphic $G_{\theta_x}(K)$-Higgs field $\varphi(z)$ is with irregular type $Q_x$. A metrized $G$-Higgs bundle $(E,\partial''_E, \phi,h)$ is with \emph{irregular type} $\boldsymbol{Q}$ if the corresponding merohoric $\mathcal{G}_{\boldsymbol\theta}$-Higgs torsor (under the functor $\Xi_{\rm Higgs}$) is with irregular type $\boldsymbol{Q}$.
	
	A merohoric $\mathcal{G}_{\boldsymbol\theta}$-connection $(\mathcal{V},\nabla)$ is with \emph{irregular type} $\boldsymbol{Q}$ if for each puncture $x \in \boldsymbol D$, the meromorphic $G_{\theta_x}(K)$-connection $\nabla(z)$ is with irregular type $Q_x$. A metrized $G$-connection $(V,d',h)$ is with \emph{irregular type} $\boldsymbol{Q}$ if the corresponding merohoric $\mathcal{G}_{\boldsymbol\theta}$-connection (under the functor $\Xi_{\rm Conn}$) is with irregular type $\boldsymbol{Q}$.
\end{defn}

Now we are ready to state the main results in this section: Theorem \ref{thm_KH_corr}, Corollary \ref{cor_ana_dR_Dol} and Theorem \ref{thm_alg_dR_Dol}.

\begin{thm}[Kobayashi--Hitchin Correspondence]\label{thm_KH_corr}
	Let $(E,\partial''_E,\phi,h_0)$ be a metrized $\boldsymbol\theta$-adapted $G$-Higgs bundle  of degree zero with irregular type $\boldsymbol{Q}$, where $h_0$ is a given adapted metric. Then, $(E,\partial''_E,\phi)$ admits an adapted harmonic metric $h$, which is quasi-isometric to $h_0$, if and only if $(E,\partial''_E,\phi,h_0)$ is $R_{h_0}$-polystable. Moreover, the metric $h$ is unique up to automorphisms of Higgs bundles. Similarly, a metrized $\boldsymbol\theta$-adapted $G$-connection $(V,d',h)$  of degree zero with irregular type $\boldsymbol{Q}$ admits an adapted harmonic metric $h$ if and only if $(V,d',h)$ is $R_{h_0}$-polystable.
\end{thm}

As a corollary of this theorem, we establish a correspondence between $R_h$-stable metrized $G$-Higgs bundles of degree zero and $R_h$-stable metrized $G$-connections of degree zero. Before we state the result, we fix the following notations:
\begin{itemize}
	\item $\widetilde{\boldsymbol{Q}}=\{\widetilde{Q}_x, x \in \boldsymbol{D}\}$ and $\boldsymbol{Q}=\{Q_x, x \in \boldsymbol{D}\}$ are two collections of irregular types;
	\item $\boldsymbol\alpha = \{\alpha_x, x \in \boldsymbol{D}\}$ and $\boldsymbol\beta =\{\beta_x, x \in \boldsymbol{D}\}$ are two collections of weights;
	\item $\phi_{\boldsymbol\alpha}=\{\phi_{\alpha_x}, x \in \boldsymbol{D}\}$ and $d'_{\boldsymbol\beta}=\{d'_{\beta_x}, x \in \boldsymbol{D}\}$ are two collections of elements in $\mathfrak{g}$, which are regarded as \emph{residues},
\end{itemize}
and then, we define the following two categories:

\begin{itemize}
\item $\mathcal{C}_{\rm Dol}(X_{\boldsymbol D},G,\boldsymbol\alpha, \phi_{\boldsymbol\alpha},\widetilde{\boldsymbol{Q}})$: the category of $R_h$-stable metrized $\boldsymbol\alpha$-adapted $G$-Higgs bundles of degree zero on $X_{\boldsymbol D}$ with irregular types $\widetilde{\boldsymbol{Q}}$ and the Levi factors of residues of the Higgs field are $\phi_{\boldsymbol\alpha}$ at punctures;
\item $\mathcal{C}_{\rm dR}(X_{\boldsymbol D},G,\boldsymbol\beta,d'_{\boldsymbol\beta},\boldsymbol{Q})$: the category of $R_h$-stable metrized $\boldsymbol\beta$-adapted $G$-connections of degree zero on $X_{\boldsymbol D}$ with irregular types $\boldsymbol{Q}$ and the Levi factors of residues of the connection are $d'_{\boldsymbol\beta}$ at punctures.
\end{itemize}

For each puncture $x \in \boldsymbol{D}$, we collect the above data in the following table
\begin{center}
	\begin{tabular}{|c|c|c|}
		\hline
		\rule{0pt}{2.6ex} &\ \ \  Dolbeault\ \ \ &\ \ \ de Rham\ \ \  \\[0.5ex]
		\hline
		\rule{0pt}{2.6ex} weights & $\alpha_x$ & $\beta_x$  \\[0.7ex]
		\hline
		\rule{0pt}{2.6ex} residues & $\phi_{\alpha_x}$ & $d'_{\beta_x}$ \\[0.7ex]
		\hline
		\rule{0pt}{2.6ex} irregular types & $\widetilde{Q}_x$ & $Q_x$ \\[0.5ex]
		\hline
	\end{tabular}
\end{center}
\begin{cor}\label{cor_ana_dR_Dol}
	The categories $\mathcal{C}_{\rm Dol}(X_{\boldsymbol D},G,\boldsymbol\alpha, \phi_{\boldsymbol\alpha},\widetilde{\boldsymbol{Q}})$ and $\mathcal{C}_{\rm dR}(X_{\boldsymbol D},G,\boldsymbol\beta,d'_{\boldsymbol\beta},\boldsymbol{Q})$ are equivalent with respect to the following table of data
	\begin{center}
		\begin{tabular}{|c|c|c|}
			\hline
			\rule{0pt}{2.6ex} & Dolbeault &\ \ \ de Rham\ \ \  \\[0.5ex]
			\hline
			\rule{0pt}{2.6ex} weights & $\frac{1}{2}(s_{\beta_x}+\bar s_{\beta_x})$ & $\beta_x$  \\[0.7ex]
			\hline
			\rule{0pt}{2.6ex} residues & $\ \frac{1}{2}(s_{\beta_x}-\beta_x)+(Y_{\beta_x}-H_{\beta_x}+X_{\beta_x})\ $ & $s_{\beta_x}+Y_{\beta_x}$ \\[0.7ex]
			\hline
			\rule{0pt}{2.6ex} irregular types & $\frac{1}{2}Q_x$ & $Q_x$ \\[0.7ex]
			\hline
		\end{tabular}
	\end{center}
	where $d'_{\beta_x} = s_{\beta_x}+Y_{\beta_x}$ is the Jordan decomposition, $Y_{\beta_x}$ is nilpotent part and $(X_{\beta_x},Y_{\beta_x},H_{\beta_x})$ is an appropriate $\mathfrak{sl}_2$-triple.
\end{cor}

The key part of the proof of Theorem \ref{thm_KH_corr}, Corollary \ref{cor_ana_dR_Dol} is to construct a model metric $h_0$ as defined in \S\ref{subsect_prel_ana} and study the relation between Higgs bundles and integrable connections locally. After the local study, the proofs are the same as in \cite[Theorem 5.1]{BGM}, \cite[Theorem 6.2]{HKSZ} and \cite[Lemma 6.7]{HKSZ}. Thus, in the next subsection, we only give the construction of a model metric $h_0$ and the explicit correspondence for the local data.

\begin{rem}
In the papers \cite{Boalch2012,Boalch2018,CDDNP2020,DDP2018}, the authors consider the connection in the form $d-B(z)\frac{dz}{z}$, and thus the corresponding irregular types for Dolbeault and de Rham are $-\frac{1}{2}Q$ and $Q$, respectively.
\end{rem}

For the algebraic side, we define
\begin{itemize}
\item $\mathcal{C}_{\rm Dol}(X,\mathcal{G}_{\boldsymbol\alpha}, \varphi_{\boldsymbol\alpha},\widetilde{\boldsymbol{Q}})$: the category of $R$-stable merohoric $\mathcal{G}_{\boldsymbol\alpha}$-Higgs torsors of degree zero on $X$ with irregular types $\widetilde{\boldsymbol{Q}}$ and the Levi factors of residues of the Higgs field are $\varphi_{\boldsymbol\alpha}$ at punctures;

\item $\mathcal{C}_{\rm dR}(X,\mathcal{G}_{\boldsymbol\beta},\nabla_{\boldsymbol\beta},\boldsymbol{Q})$: the category of $R$-stable merohoric $\mathcal{G}_{\boldsymbol\beta}$-connections of degree zero on $X$ with irregular types $\boldsymbol{Q}$ and the Levi factors of residues of the connection are $\nabla_{\boldsymbol\beta}$ at punctures.
\end{itemize}

Based on the functors $\Xi_{\rm Dol}$ and $\Xi_{\rm dR}$, we have the desired correspondence for algebraic sides.
\begin{thm}\label{thm_alg_dR_Dol}
	The categories $\mathcal{C}_{\rm Dol}(X,\mathcal{G}_{\boldsymbol\alpha}, \varphi_{\boldsymbol\alpha},\widetilde{\boldsymbol{Q}})$ and $\mathcal{C}_{\rm dR}(X,\mathcal{G}_{\boldsymbol\beta},\nabla_{\boldsymbol\beta},\boldsymbol{Q})$ are equivalent, where $\varphi_{\boldsymbol\alpha} = \phi_{\boldsymbol\alpha}$, $\nabla_{\boldsymbol\beta}=d'_{\boldsymbol\beta}$ and the data are exactly the same as the data in the table of Corollary \ref{cor_ana_dR_Dol}.
\end{thm}

\begin{proof}
The idea of the proof is given in the following diagram
\begin{center}
	\begin{tikzcd}
	 \mathcal{C}_{\rm Dol}(X_{\boldsymbol D},G,\boldsymbol\alpha, \phi_{\boldsymbol\alpha},\widetilde{\boldsymbol{Q}}) \arrow[rr, equal, "\text{ Corollary } \ref{cor_ana_dR_Dol}"] \arrow[d, "\Xi_{\rm Dol}"] & & \mathcal{C}_{\rm dR}(X_{\boldsymbol D},G,\boldsymbol\beta,d'_{\boldsymbol\beta},\boldsymbol{Q}) \arrow[d, "\Xi_{\rm dR}"] \\
	\mathcal{C}_{\rm Dol}(X,\mathcal{G}_{\boldsymbol\alpha}, \varphi_{\boldsymbol\alpha},\widetilde{\boldsymbol{Q}}) \arrow[rr, dash, dotted] & & \mathcal{C}_{\rm dR}(X,\mathcal{G}_{\boldsymbol\beta},\nabla_{\boldsymbol\beta},\boldsymbol{Q}).
	\end{tikzcd}
\end{center}
Corollary \ref{cor_ana_dR_Dol} gives the equivalence of two categories of analytic objects. The functors $\Xi_{\rm Dol}$ and $\Xi_{\rm dR}$ transfer analytic objects to algebraic objects. To prove this theorem, it is enough to show that the functors $\Xi_{\rm Dol}$ and $\Xi_{\rm dR}$ provide equivalences. We only give the proof for the functor $\Xi_{\rm Dol}$. By the uniqueness of the harmonic metric (Theorem \ref{thm_KH_corr}), the functor $\Xi_{\rm Dol}$ is injective on the objects. Now we consider the surjectivity. Take a merohoric $\mathcal{G}_{\boldsymbol\theta}$-Higgs torsor $(\mathcal{E},\varphi)$ in the category $\mathcal{C}_{\rm Dol}(X,\mathcal{G}_{\boldsymbol\alpha}, \varphi_{\boldsymbol\alpha},\widetilde{\boldsymbol{Q}})$. We obtain a holomorphic $(\boldsymbol\theta,\boldsymbol{n})$-adapted $G$-Higgs bundle $(E,\partial''_E,\phi)$ with irregular type $\widetilde{\boldsymbol{Q}}$. By the construction of the model metric in \S\ref{subsect_local_study}, we obtain an adapted metric $h_0$ on $E$ (by partition of unity). Therefore, we obtain an element $(E,\partial''_E,\phi,h_0)$ in the category $\mathcal{C}_{\rm Higgs}(X_{\boldsymbol D},G,\boldsymbol\alpha, \phi_{\boldsymbol\alpha},\widetilde{\boldsymbol{Q}})$ such that it is $R_h$-stable and of degree zero. By Theorem \ref{thm_KH_corr} again, we obtain a harmonic metric $h$ on $E$, and therefore an element $(E,\partial''_E,\phi,h)$ in the category $\mathcal{C}_{\rm Dol}(X_{\boldsymbol D},G,\boldsymbol\alpha, \phi_{\boldsymbol\alpha},\widetilde{\boldsymbol{Q}})$. This finishes the proof of this theorem.
\end{proof}

\subsection{Local Study}\label{subsect_local_study}
In this subsection, we give an appropriate model metric $h_0$ by generalizing Sabbah's result \cite{Sab} to general complex reductive groups. Let $x\in\boldsymbol{D}$ be a point in the divisor, and denote by $\mathbb{D}^*$ a punctured unit disc around $x$ with local coordinate $z$. To simplify the notation, the weight is denoted by $\beta$ and the irregular type is denoted by $Q$. Given a meromorphic $G$-connection with irregular type $Q$, one may take a local holomorphic basis $\mathbf{e}$ so that the connection $D$ has the form
\begin{align*}
d +  B(z)\frac{dz}{z} & = d + (B_{-n} z^{-n} + \dots + B_{0})\frac{dz}{z} \\ 
& = d + (B_0 + z Q'(z))\frac{dz}{z},
\end{align*}
with $B_{0} \in\mathfrak{l}_\beta$, and $B_{-i} \in\mathfrak{t}$ for $i \geq 1$. Let $\beta$ be a weight for the connection. We apply the Jordan decomposition for $B_{0}$ as
\begin{align*}
B_0=s_\beta+Y_\beta,
\end{align*}
where $s_\beta\in\mathfrak{t}$ is the semisimple part and $Y_\beta$ the nilpotent part. Completing $Y_\beta$ into a natural $\mathfrak{sl}_2$-triple $(X_\beta,H_\beta,Y_\beta)$ so that $X_\beta$ is the ``hermitian dual'' of $Y_\beta$ and $H_\beta=[X_\beta,Y_\beta]\in\mathfrak{t}$. They satisfy the relations $[H_\beta,X_\beta]=2X_\beta, [H_\beta,Y_\beta]=-2Y_\beta$. Then, we can write 
\begin{align*}
D\mathbf{e}=\mathbf{e}\cdot(s_\beta+Y_\beta+zQ'(z))\frac{dz}{z}.
\end{align*}
Decompose the exterior differential operator into different types as $d=\partial'_0+\partial''_0$. The model metric we consider is denoted by $h_0$ so that the new basis $\mathbf{e}_0:=\mathbf{e}\cdot g_0$ is orthonormal, where
\begin{align*}
g_0=|z|^{-\beta}(-\ln|z|^2)^{-H_\beta/2}e^{X_\beta}.
\end{align*}
Therefore, under the trivialization $\mathbf{e}$, the model metric $h_0$ has the form
\begin{align}\label{model metric}
h_0=|z|^{2\beta}(-\ln|z|^2)^{H_\beta/2}e^{-Y_\beta}e^{-X_\beta}(-\ln|z|^2)^{H_\beta/2},
\end{align}
and
\begin{align}\label{residue-con}
|\mathbf{e}|^2_{h_0}=|z|^{2\beta}(-\ln|z|^2)^{H_\beta/2}e^{-Y_\beta}e^{-X_\beta}(-\ln|z|^2)^{H_\beta/2},
\end{align}
this characterizes the jump, i.e. weight $\beta$ of the connection $D$.

\begin{lem}[Lemma 3.2 in \cite{Sab}]
	The following formulas will be used frequently
	\begin{align*}
	\partial_z\Big((-\ln|z|^2)^{\pm H_\beta/2}\Big)&=\pm\frac{H_\beta}{2\ln|z|^2}(-\ln|z|^2)^{\pm H_\beta/2}\cdot z^{-1};\\
	\partial_{\bar z}\Big((-\ln|z|^2)^{\pm H_\beta/2}\Big)&=\pm\frac{H_\beta}{2\ln|z|^2}(-\ln|z|^2)^{\pm H_\beta/2}\cdot {\bar z}^{-1};\\
	(-\ln|z|^2)^{\pm H_\beta/2}X_\beta(-\ln|z|^2)^{\mp H_\beta/2}&=(-\ln|z|^2)^{\pm 1}X_\beta;\\
	(-\ln|z|^2)^{\pm H_\beta/2}Y_\beta(-\ln|z|^2)^{\mp H_\beta/2}&=(-\ln|z|^2)^{\mp 1}Y_\beta;\\
	e^{X_\beta}H_\beta e^{-X_\beta}&=H_\beta-2X_\beta;\\
	e^{Y_\beta}H_\beta e^{-Y_\beta}&=H_\beta+2Y_\beta;\\
	e^{X_\beta}Y_\beta e^{-X_\beta}&=Y_\beta+H_\beta-X_\beta,
	\end{align*}
	which in turn imply
	\begin{align*}
	e^{-X_\beta}H_\beta e^{X_\beta}&=H_\beta+2X_\beta;\\
	e^{-Y_\beta}H_\beta e^{Y_\beta}&=H_\beta-2Y_\beta;\\
	e^{-X_\beta}Y_\beta e^{X_\beta}&=Y_\beta-H_\beta-X_\beta.
	\end{align*}
\end{lem}

\begin{proof}
	Here we just give for the reader a proof to the third one, and the others can be proved analogously. In fact,
	\begin{align*}
	(-\ln|z|^2)^{\pm H_\beta/2}X_\beta(-\ln|z|^2)^{\mp H_\beta/2}&=\mathrm{Ad}(e^{\pm H_\beta/2\ln(-\ln|z|^2)})X_\beta\\
	&=e^{\mathrm{ad}(\pm H_\beta/2\ln(-\ln|z|^2))}X_\beta\\
	&=(-\ln|z|^2)^{\pm 1}X_\beta.
	\end{align*}
\end{proof}

\begin{lem}
	The curvature $F_{h_0}=(D_{h_0})^2$ of the Chern connection $D_{h_0}:=\partial_{h_0}'+\partial_0''$ associated to $h_0$ and $\partial_0''$ is acceptable, namely
	\begin{align*}
	\|F_{h_0}\|_{h_0}\leq\frac{C}{|z^2|(\ln|z|^2)^2}
	\end{align*}
	holds for some constant $C>0$.
\end{lem}

\begin{proof}
	Under the trivialization $\mathbf{e}$, 
	\begin{align*}
	\partial_{h_0}'&=\partial_0'+h_0^{-1}\cdot\partial h_0 \\
	&=\partial_0'+\Big(\beta-Y_\beta+\frac{2H_\beta}{\ln|z|^2}+\frac{2X_\beta}{(\ln|z|^2)^2}\Big)\frac{dz}{z},
	\end{align*}
	so the Chern connection is
	\begin{align*}
	D_{h_0}=d+\Big(\beta-Y_\beta+\frac{2H_\beta}{\ln|z|^2}+\frac{2X_\beta}{(\ln|z|^2)^2}\Big)\frac{dz}{z},
	\end{align*}
	and the curvature $F_{h_0}(\mathbf{e})$ under the trivialization $\mathbf{e}$ has the expression
	\begin{align*}
	F_{h_0}(\mathbf{e})=(D_{h_0})^2&=-{\bar z}\partial_{\bar z}\Big(\beta-Y_\beta+\frac{2H_\beta}{\ln|z|^2}+\frac{2X_\beta}{(\ln|z|^2)^2}\Big)\frac{dz\wedge d\bar z}{|z|^2}\\
	&=\Big(\frac{2H_\beta}{(\ln|z|^2)^2}+\frac{4X_\beta}{(\ln|z|^2)^3}\Big)\frac{dz\wedge d\bar z}{|z|^2}.
	\end{align*}
	Under the new trivialization $\mathbf{e}_0$, the curvature $F_{h_0}(\mathbf{e}_0)$ as the following expression
	\begin{align*}
 F_{h_0}(\mathbf{e}_0)=g_0^{-1}F_{h_0}(\mathbf{e})g_0=\frac{2H_\beta}{(\ln|z|^2)^2}\frac{dz\wedge d\bar z}{|z|^2},
	\end{align*}
	and the acceptability condition immediately follows.
\end{proof}

As introduced in \S\ref{subsect_prel_ana}, given the metric $h_0$, the meromorphic connection $D$ determines a tuple $(E,\partial''_{h_0},\phi_{h_0},h_0)$ (see \eqref{eq_metric_Higgs}). Let $D''_{h_0} : = \partial_{h_0}''+\phi_{h_0}$. Now we will check that the pseudo-curvature $F_{D''_{h_0}} =(D''_{h_0})^2$ is trivial, which implies that the tuple $(E,\partial''_{h_0},\phi_{h_0},h_0)$ is a metrized $G$-Higgs bundle.

\begin{lem}
	We have $F_{D''_{h_0}}=0$.
\end{lem}

\begin{proof}
	Decompose $D=d_0'+d_0''$ into different types, then under the trivialization $\mathbf{e}_0$, we have
	\begin{align*}
	d_0'&=\partial_0'+g_0^{-1}\cdot(s_\beta+Y_\beta+zQ'(z))\cdot g_0\frac{dz}{z}+g_0^{-1}\partial g_0\\
	&=\partial_0'+(s_\beta+zQ'(z))\frac{dz}{z}+g_0^{-1}Y_\beta g_0\frac{dz}{z}+g_0^{-1}\partial g_0\\
	&=\partial_0'+(s_\beta+zQ'(z))\frac{dz}{z}-\frac{\beta}{2}\frac{dz}{z}-\frac{2Y_\beta-H_\beta}{2\ln|z|^2}\frac{dz}{z};\\
	d_0''&=\partial_0''+g_0^{-1}\bar\partial g_0\\
	&=\partial_0''-\frac{\beta}{2}\frac{d\bar z}{\bar z}-\frac{H_\beta+2X_\beta}{2\ln|z|^2}\frac{d\bar z}{\bar z}.
	\end{align*}
	Denote respectively by $\delta_{h_0}'$ and $\delta_{h_0}''$ the operators of type $(1,0)$ and $(0,1)$ so that both $d_0'+\delta_{h_0}''$ and $d_0''+\delta_{h_0}'$ preserve the metric $h_0$. Since $\mathbf{e}_0$ is orthonormal,
	\begin{align*}
	\partial_{h_0}'&=\frac{1}{2}(d_0'+\delta_{h_0}')=\partial_0'+\frac{1}{2}(d_0'-(\overline{d_0''})^\mathrm{T})=\partial_0'+\frac{1}{2}\Big(s_\beta+zQ'(z)+\frac{H_\beta}{\ln|z|^2}\Big)\frac{dz}{z};\\
	\partial_{h_0}''&=\frac{1}{2}(d_0''+\delta_{h_0}'')=\partial_0''+\frac{1}{2}(d_0''-(\overline{d_0'})^\mathrm{T})=\partial_0''-\frac{1}{2}\Big(\bar s_\beta+\bar z\overline{Q'(z)}+\frac{H_\beta}{\ln|z|^2}\Big)\frac{d\bar z}{\bar z};\\
	\phi_{h_0}&=\frac{1}{2}(d_0'-\delta_{h_0}')=\frac{1}{2}(d_0'+(\overline{d_0''})^\mathrm{T})=\frac{1}{2}\Big(s_\beta+zQ'(z)-\beta-\frac{2Y_\beta}{\ln|z|^2}\Big)\frac{dz}{z};\\
	\phi_{h_0}^*&=\frac{1}{2}(d_0''-\delta_{h_0}'')=\frac{1}{2}(d_0''+(\overline{d_0'})^\mathrm{T})=\frac{1}{2}\Big(\bar s_\beta+\bar{z}\overline{Q'(z)}-\beta-\frac{2X_\beta}{\ln|z|^2}\Big)\frac{d\bar z}{\bar z}.
	\end{align*}
	Write 
	\begin{align*}
	\partial_{h_0}''&=:\partial_0''-\frac{1}{2}\bar{z}\overline{Q'(z)}\frac{d\bar z}{\bar z}+K^{\mathrm{reg}}\frac{d\bar z}{\bar z},\\
	\phi_{h_0}&=:\frac{1}{2}zQ'(z)\frac{dz}{z}+M^{\mathrm{reg}}\frac{dz}{z},
	\end{align*}
	with 
	\begin{align*}
	K^{\mathrm{reg}}=-\frac{1}{2}\bar s_\beta-\frac{H_\beta}{2\ln|z|^2},\quad M^{\mathrm{reg}}=\frac{1}{2}(s_\beta-\beta)-\frac{Y_\beta}{\ln|z|^2}.
	\end{align*}
	Then
	\begin{align}\label{pseudo}
	F_{D''_{h_0}}=\partial_{h_0}''(\phi_{h_0})=-\Big(\bar{z}\partial_{\bar z}M^{\mathrm{reg}}+[K^{\mathrm{reg}},M^{\mathrm{reg}}]\Big)\frac{dz\wedge d\bar z}{|z|^2},
	\end{align}
	so we can simplify the calculation by taking $Q(z)=0$. We now have
	\begin{align*}
	\bar{z}\partial_{\bar z}M^{\mathrm{reg}}&=\frac{Y_\beta}{(\ln|z|^2)^2};\\
	[K^{\mathrm{reg}},M^{\mathrm{reg}}]&=\Big[\frac{H_\beta}{2\ln|z|^2},\frac{Y_\beta}{\ln|z|^2}\Big]=-\frac{Y_\beta}{(\ln|z|^2)^2},
	\end{align*}
	and substituting into \eqref{pseudo}, we have $F_{D''_{h_0}}=0$.
\end{proof}

Note that
\begin{align*}
Q'(z)=\sum_{i=1}^n\frac{B_{-i}}{z^{i+1}}
\end{align*}
with each $B_i\in\mathfrak{t}$. Consider the new orthonormal basis
\begin{align*}
\mathbf{e}_1:=\mathbf{e}_0\cdot g_1
\end{align*}
with 
\begin{align*}
g_1=\exp\Big(\sum_{i=1}^n\Big(\frac{B_{-i}}{2i}z^{-i}-\frac{\bar{B}_{-i}}{2i}\bar{z}^{-i}\Big)\Big),
\end{align*}
and then under the new trivialization $\mathbf{e}_1$, we have
\begin{align*}
g_1^{-1}(\partial_{h_0}'+\partial_{h_0}'')g_1+g_1^{-1}dg_1 &= d+\frac{1}{2}\Big(s_\beta+\frac{H_\beta}{\ln|z|^2}\Big)\frac{dz}{z}-\frac{1}{2}\Big(\bar s_\beta+\frac{H_\beta}{\ln|z|^2}\Big)\frac{d\bar{z}}{\bar z};\\
g_1^{-1}(\phi_{h_0}+\phi_{h_0}^*)g_1 &= \phi_{h_0}+\phi_{h_0}^*.
\end{align*}
Namely the holomorphic structure of the inducing Higgs bundle has the expression
\begin{align*}
\partial''=\partial_0''-\frac{1}{2}\Big(\bar s_\beta+\frac{H_\beta}{\ln|z|^2}\Big)\frac{d\bar{z}}{\bar z}.
\end{align*}
Finally, we take the following holomorphic basis
\begin{align*}
\mathbf{e}_2=\mathbf{e}_1\cdot g_2
\end{align*}
with 
\begin{align*}
g_2=|z|^{\bar s_\beta}(-\ln|z|^2)^{H_\beta/2}e^{X_\beta},
\end{align*}
and then 
\begin{align}\label{weight-higgs}
|\mathbf{e}_2|_{h_0}^2=|z|^{s_\beta+\bar s_\beta}e^{Y_\beta}(-\ln|z|^2)^{H_\beta}e^{X_\beta}.
\end{align}
Under the trivialization $\mathbf{e}_2$, the Higgs field has the form
\begin{align}\label{residue-higgs}
\phi=g_2^{-1}\phi_{h_0}g_2=\left[ \frac{1}{2}(s_\beta-\beta+zQ'(z)) +(Y_\beta-H_\beta+X_\beta) \right]\frac{dz}{z}.
\end{align}
In conclusion, we have the following table between Dolbeault and de Rham data:
\bigskip
\begin{center}
	\begin{tabular}{|c|c|c|}
		\hline
		\rule{0pt}{2.6ex} & Dolbeault &\ \ \ de Rham\ \ \  \\[0.5ex]
		\hline
		\rule{0pt}{2.6ex} weights & $\frac{1}{2}(s_\beta+\bar s_\beta)$ & $\beta$  \\[0.7ex]
		\hline
		\rule{0pt}{2.6ex} residues $\backslash$ monodromies & $\ \frac{1}{2}(s_\beta-\beta)+(Y_\beta-H_\beta+X_\beta)\ $ & $s_\beta+Y_\beta$ \\[0.7ex]
		\hline
		\rule{0pt}{2.6ex} irregular types & $\frac{1}{2}Q$ & $Q$ \\[0.7ex]
		\hline
	\end{tabular}
\end{center}
\medskip
In particular, when $Y_\beta=0$, we recover Biquard--Boalch's data \cite{BB}:
\medskip
\begin{center}
	\begin{tabular}{|c|c|c|}
		\hline
		\rule{0pt}{2.6ex} & \ \ Dolbeault\ \ & \ \ \ de Rham\ \ \  \\[0.5ex]
		\hline
		\rule{0pt}{2.6ex} weights & $\ \ \frac{1}{2}(s_\beta+\bar s_\beta)\ \ $ & $\beta$  \\[0.7ex]
		\hline
		\rule{0pt}{2.6ex} residues $\backslash$ monodromies & $\frac{1}{2}(s_\beta-\beta)$ & $s_\beta$ \\[0.7ex]
		\hline
		\rule{0pt}{2.6ex} irregular types & $\frac{1}{2}Q$ & $Q$ \\[0.7ex]
		\hline
	\end{tabular}
\end{center}

\section{de Rham Category vs. Betti Category}\label{sect_dR_Betti}
In this section, we give the correspondence between merohoric connections and filtered Stokes $G$-local systems. By introducing a notion of  stability condition for filtered Stokes $G$-local systems (Definition \ref{defn_stab_betti}), we prove a one-to-one correspondence between $R$-stable merohoric connections of degree zero and $R$-stable filtered Stokes $G$-local systems of degree zero (Theorem \ref{thm_dR_Betti}).

\subsection{Stokes Local System}\label{subsect_prel_stokes}

We first briefly review Stokes $G$-local systems and details can be found in \cite{Boalch2014}. In this paper, a Stokes local system will be actually referring to the corresponding Stokes representation. Let $(X,\boldsymbol{D})$ be as before, and let $\boldsymbol{n}:=\{n_x, x \in \boldsymbol{D}\}$ be a set of positive integers. Let $\boldsymbol{Q}=\{Q_x, x \in \boldsymbol{D}\}$ be a set of irregular types so that $\deg(Q_x)=n_x-1$ for each $x \in \boldsymbol{D}$. Let $(E,D)$ be a meromorphic $G$-connection on $(X,\boldsymbol{D})$ with irregular type $\boldsymbol{Q}$. Let $H_x \subseteq G$ be the stabilizer of $Q_x$ under the adjoint action, i.e. $H_x=Z_G(Q_x)$. Denote by $T_x$ the maximal torus in $H_x$. Let $\widehat{X}$ be the real oriented blow-up of $X$ at each $x\in\boldsymbol{D}$. It is the surface obtained from $X$ by replacing each $x\in\boldsymbol{D}$ with an oriented boundary circle $\partial_x$, and points of $\partial_x$ correspond to real oriented directions emanating from $x$. 

We first work locally around one puncture $x$, and to simplify the notation, we omit the subscript $x$. Each root $r \in \mathcal{R}$ determines a meromorphic function 
\begin{align*}
q_r (z):=r(Q(z)),
\end{align*}
and an \emph{anti-Stokes direction} (or \emph{singular direction}) supported by $r$ is a direction $d\in\partial$ so that the meromorphic function $\exp(q_r (z))$ has maximal decay as $z\to0$ along the direction $d$. If we denote by $\frac{c_r}{z^k}$ the most singular term of $q_r (z)$, then along an anti-Stokes direction $d\in\partial$, the term $\frac{c_r}{z^k}$ is real and negative. Denote by $\mathbb{A} \subset\partial$ the set of all anti-Stokes directions on $\partial$. Note that when we rotate an anti-Stokes direction $d\in\mathbb{A}$ by an angle $\pi/k$, the corresponding direction on $\partial$ is still an anti-Stokes direction. This implies that $\mathbb{A}$ has $\pi/k$ rotational symmetry and thus $l:=\#(\mathbb{A})/2k$ is an integer. In particular, if the leading coefficient $A_{n-1}$ of the irregular type $Q$ is regular semisimple, then $\# \mathbb{A}/2(n-1)$ is an integer. Any $l$  consecutive anti-Stokes direction of $\mathbb{A}$ is called a \emph{half-period}.

For an anti-Stokes direction $d\in\mathbb{A}$, denote by $\mathcal{R}(d)\subset \mathcal{R}$ the set of all roots that support $d$, then the following product map 
\begin{align}\label{sto-group}
\phi: \prod_{r \in\mathcal{R}(d)}U_r \longrightarrow G
\end{align}
is an algebraic isomorphism of spaces (not as groups) into its image, where each $U_r:=\exp(\mathfrak{g}_r)\subset G$ is the unipotent group associated to $r$. The image of $\phi$ is called the \emph{Stokes group} determined by $d$, and we denote it by $\mathbb{S}\mathrm{to}_d$. For an irregular type $Q$, we define the \emph{space of Stokes data} as
\begin{align*}
\mathbb{S}\mathrm{to}(Q):=\prod_{d\in\mathbb{A}}\mathbb{S}\mathrm{to}_d.
\end{align*}
It follows from \cite{Boalch2014} that the image of \eqref{sto-group} does not depend on the order of the product, so it is a well-defined map.

The Stokes data $\mathbb{S}\mathrm{to}(Q)$ can be characterized via unipotent radicals of $G$ defined by half-periods \cite[\S 7]{Boalch2014}. By marking all the $\#\mathbb{A}_x=2kl$ anti-Stokes directions in consecutive order as (by choosing an arbitrary one as $d_1$)
\begin{align*}
d_1, \cdots, d_{2kl},
\end{align*}
the half-period $\boldsymbol{d}_+:=\{d_1,\cdots,d_l\}$ yields a parabolic subgroup $P_+\subset G$ which contains $H$ as a Levi subgroup, and the opposite parabolic $P_-\subset G$ with Levi subgroup $H$ is associated to another half-period $\boldsymbol{d}_-:=\{d_{l+1},\cdots,d_{2l}\}$. Denote respectively by $U_+$ and $U_-$ the unipotent radicals of $P_+$ and $P_-$. Then, we get
\begin{align*}
\mathbb{S}\mathrm{to}(Q)\xrightarrow{\cong} (U_+\times U_-)^k.
\end{align*}
Based on this property, we define
\begin{align*}
\mathcal{A}(Q): = G \times H \times \mathbb{S}{\rm to}(Q) \cong G \times H \times (U_+\times U_-)^k.
\end{align*}

Now we move to the global picture. Suppose that there are $m$ points in $\boldsymbol{D}$, i.e. $\boldsymbol{D}=\{x_1,\dots,x_m\}$. For each $x\in\boldsymbol{D}$, draw a concentric circle (called a \emph{halo}) on $\widehat{X}$ near $\partial_x$, and then puncture the halo at $\#\mathbb{A}_x$ distinct points in accordance with the $\#\mathbb{A}_x$ anti-Stokes directions so that all the $\#\mathbb{A}_x$ auxiliary small cilia between each anti-Stokes direction and its nearby puncture do not cross. Denote by $\mathbb{H}_x$ the region between the halo and $\partial_x$ (as a tubular neighborhood of $\partial_x$). The resulting punctured curve $X_{\boldsymbol{Q}}\subset\widehat{X}$ is called the \emph{irregular curve} determined by the irregular type $\boldsymbol{Q}$. Note that the irregular types defined in this paper (see Definition \ref{defn_irregular_type}) are unramified, but the irregular curves can be defined even for ramified irregular types (see {\cite[Remark 8.6]{Boalch2014}} and {\cite[\S3.5]{BY15}} for the description of ramified irregular types).

Now for each boundary circle $\partial_x$ of $X_{\boldsymbol{Q}}$, we choose a base point $b_x$, and denote by $\boldsymbol{b}:=\{b_x, x\in\boldsymbol{D}\}$ the set of base points. Let $\Pi:=\Pi_1(X_{\boldsymbol{Q}},\boldsymbol{b})$ be the fundamental groupoid of $X_{\boldsymbol{Q}}$ with $\boldsymbol{b}$ as the set of base points. Here is an explicit description of the generators of $\Pi$:
\begin{enumerate}
	\item $\alpha_1,\beta_1,\cdots,\alpha_g,\beta_g$, the loops determined by the genus of $X$;
	\item for each $x\in\boldsymbol{D}$, the simple closed loop $\gamma_x$ based at $b_x$ going once around $\partial_x$;
	\item for each $d\in\mathbb{A}_x$, the loop $\tilde{\gamma}_d$ based at $b_x$ going once around the nearby puncture $\tilde{d}$ of $d$ so that $\tilde{d}$ is the only puncture inside $\tilde{\gamma}_d$;
	\item from distinct base points $b_x$ and $b_{x'}$, the simple path $\gamma_{xx'}$ connecting them.
\end{enumerate}
\noindent Let $\mathrm{Hom}(\Pi,G)$ be the space of $G$-representations of $\Pi$. 

An element $\rho\in\mathrm{Hom}(\Pi,G)$ is called a \emph{Stokes $G$-representation} or a \emph{Stokes $G$-local system} on $X_{\boldsymbol{Q}}$ if for each $x\in\boldsymbol{D}$ and $d\in\mathbb{A}_x$, we have $\rho(\gamma_x)\in H_x$ and $\rho(\tilde{\gamma}_d)\in\mathbb{S}\mathrm{to}_d$. Denote by $\mathrm{Hom}_{\mathbb{S}}(\Pi,G)$ the space of all Stokes $G$-representations. Recall that $H_x$ is the stabilizer of $Q_x$ for $x \in \boldsymbol{D}$, and we define
\begin{align*}
\boldsymbol{H}:=\prod_{x \in \boldsymbol{D}} H_x = H_{x_1} \times \dots \times H_{x_m}.
\end{align*} 
The space of Stokes $G$-local systems $\mathrm{Hom}_{\mathbb{S}}(\Pi,G)$ is a smooth affine variety and furthermore, a quasi-Hamiltonian $\boldsymbol{H}$-space \cite[Theorem 8.2]{Boalch2014}, where the $\boldsymbol{H}$-action is given by
\begin{align*}
((k_1,\cdots,k_m)\cdot\rho)(\gamma):=k_j\cdot\rho(\gamma)\cdot k_i^{-1},
\end{align*}
where $\gamma$ is any path in $\Pi$ from $b_{x_i}$ to $b_{x_j}$. 

Here is an equivalent description of ${\rm Hom}_{\mathbb{S}}(\Pi,G)$. For each $x_i\in\boldsymbol{D}$, we define 
\begin{align*}
\mathcal{A}(Q_{x_i})=G\times H_{x_i}\times\mathbb{S}\mathrm{to}(Q_{x_i}). 
\end{align*}
Then, the space of Stokes $G$-local systems $\mathrm{Hom}_{\mathbb{S}}(\Pi,G)$ can be characterized by a closed subvariety $\mathcal{U}_{\boldsymbol{Q}}$ of the affine variety $\mathcal{R}_{\boldsymbol{Q}}:=(G\times G)^g\times\mathcal{A}(Q_{x_1})\times\cdots\times\mathcal{A}(Q_{x_m})$ (see also \cite[(2.36)]{Witten2007} for the case of one irregular point):
\begin{align}\label{space-Sto}
\begin{aligned}
\mathcal{U}_{\boldsymbol{Q}}&:=\Big\{((A_i,B_i)_{1\leq i\leq g}, (C_{x_i},h_{x_i},S_{x_i,j})_{1\leq i\leq m,1\leq j\leq\#\mathbb{A}_{x_i}})\in\mathcal{R}_{\boldsymbol{Q}}\ :\ \\ \
&\hspace{10em}\prod_{i=1}^g[A_i,B_i]\prod_{i=1}^m\Big(C_{x_i}^{-1}h_{x_i}S_{x_i,\#\mathbb{A}_{x_i}}\cdots S_{x_i,1}C_{x_i}\Big)=\mathrm{Id}
\Big\},
\end{aligned}
\end{align}
the $\boldsymbol{H}$-action on $\mathrm{Hom}_{\mathbb{S}}(\Pi,G)$ is represented by the $(G\times\boldsymbol{H})$-action on $\mathcal{U}_{\boldsymbol{Q}}$ via
\begin{align*}
(g,(k_{x_i}&)_{1\leq i\leq m})\cdot((A_i,B_i)_{1\leq i\leq g}, (C_{x_i},h_{x_i},S_{x_i,j})_{1\leq i\leq m,1\leq j\leq\#\mathbb{A}_{x_i}}):=\\
&\hspace{5em}((gA_ig^{-1},gB_ig^{-1})_{1\leq i\leq g}, (k_{x_i}C_{x_i}g^{-1},k_{x_i} h_{x_i} k^{-1}_{x_i},k_{x_i}S_{x_i,j}k_{x_i}^{-1})_{1\leq i\leq m,1\leq j\leq\#\mathbb{A}_{x_i}}).
\end{align*}

\begin{thm}[Theorem A.3 and Corollary A.4 in \cite{Boalch2014}]\label{thm_Boal_equiv}
Let $G$ be a complex reductive group. There is an equivalence between the category of meromorphic $G$-connections on $X$ with irregular type $\boldsymbol{Q}$ and the category of Stokes $G$-local systems on $X_{\boldsymbol{Q}}$.
\end{thm}
The theorem for the case of ${\rm GL}_n(\mathbb{C})$ is well-known (see \cite[Appendix]{Su2019} for instance). Here is the table for the relation of residues (monodromies) and irregular types between meromorphic $G$-connections and Stokes $G$-local systems around a particular puncture $x$.
\begin{center}
	\begin{tabular}{|c|c|c|}
		\hline
		\rule{0pt}{2.6ex} &\ \ \  de Rham\ \ \ & Betti  \\[0.5ex]
		\hline
		\rule{0pt}{2.6ex} residues $\backslash$ monodromies & $\nabla_{\beta_x}$ & $h_x = \exp(-2\pi \sqrt{-1}\nabla_{\beta_x})$ \\[0.7ex]
		\hline
		\rule{0pt}{2.6ex} irregular types & $Q_x$ & $Q_x$ \\[0.7ex]
		\hline
	\end{tabular}
\end{center}

Now we review the definition of irreducibility for Stokes $G$-local systems \cite[\S 9]{Boalch2014}. For each puncture $x \in \boldsymbol{D}$, we fix a parabolic group $P_x$. We say that the collection $\boldsymbol{P}=\{P_{x_i}, 1 \leq i \leq m\}$ is \emph{compatible} with an element $\rho \in {\rm Hom}_{\mathbb{S}}(\Pi,G)$, if we have
\begin{align*}
\rho(\gamma) P_{x_i} \rho(\gamma)^{-1} = P_{x_j}
\end{align*}
for any path $\gamma$ in $X_{\boldsymbol{Q}}$ from $b_i$ to $b_j$ (for arbitrary $i,j$). A collection of parabolic subgroups $\boldsymbol{P}$ is called \emph{invariant}, if the maximal torus $T_x$ (of $H_x$) is included in $P_x$ for each $x \in \boldsymbol{D}$. Furthermore, it is called \emph{proper}, if each $P_x$ is a proper parabolic subgroup of $G$. Based on these definitions, Boalch introduced the definition of an \emph{irreducible Stokes $G$-local system}.

\begin{defn}\label{defn_Boal_irre_rep}
A Stokes $G$-local system $\rho \in {\rm Hom}_{\mathbb{S}}(\Pi,G)$ is called \emph{reducible} if there is an invariant proper collection $\boldsymbol{P}$ of parabolic subgroups compatible with $\rho$. Otherwise, it is called \emph{irreducible}. 
\end{defn}

\subsection{Filtered Stokes $G$-local Systems and Betti Category}\label{subsect_fil_stokes}

Let $\boldsymbol\theta=\{\theta_x, x \in \boldsymbol{D}\}$ be a collection of weights. Denote by $P_{\theta_x} \subseteq G$ the corresponding parabolic subgroup determined by the weight $\theta_x$, and let $L_{\theta_x}$ be the Levi component of $P_{\theta_x}$. Note that an element $\rho \in {\rm Hom}_{\mathbb{S}}(\Pi,G)$ can be regarded as a tuple 
\begin{align*}
	((A_i,B_i)_{1\leq i\leq g}, (C_{x_i},h_{x_i},S_{x_i,j})_{1\leq i\leq m,1\leq j\leq\#\mathbb{A}_{x_i}})\in\mathcal{R}_{\boldsymbol{Q}}.
\end{align*}

\begin{defn}
A \emph{$\boldsymbol\theta$-filtered Stokes $G$-local system} is a representation $\rho \in {\rm Hom}_{\mathbb{S}}(\Pi,G)$ such that the formal monodromy $h_x$ is an element in $P_{\theta_x}$ for each $x \in \boldsymbol{D}$.
\end{defn}

\begin{rem}
Given an irregular type $Q_x$, we define its stabilizer group $H_x$. By properties of Stokes $G$-representations, the formal monodromy $h_x$ is included in $H_x$. Now for a $\boldsymbol\theta$-filtered Stokes $G$-local system, we require $h_x \in P_{\theta_x}$. Thus, the formal monodromy actually lies in $H_x \cap P_{\theta_x}$. This condition is also expected from the side of merohoric $\mathcal{G}_{\boldsymbol\theta}$-connections. Proposition \ref{prop_canonical_form_parah} shows that the coefficient $B_0$ (the residue of a meromorphic connection $d+B(z)\frac{dz}{z}$) is in $\mathfrak{p}_{\theta_x}$ and also commutes with $B_{i}$ for $i \leq -1$. This property is equivalent to saying that $\exp(-2\pi \sqrt{-1}B_0)$ is in $H_x \cap P_{\theta_x}$.
\end{rem}

\begin{lem}\label{lem_conn_and_rep_trivial}
If the weights $\boldsymbol\theta$ are trivial, i.e. $\theta_x = 0$ for any $x \in \boldsymbol{D}$, then the category of merohoric $\mathcal{G}_{\boldsymbol\theta}$-connections on $X$ with irregular type $\boldsymbol{Q}$ is equivalent to the category of $\boldsymbol\theta$-filtered Stokes $G$-local systems on $X_{\boldsymbol{Q}}$.
\end{lem}

\begin{proof}
When the weights are trivial, merohoric $\mathcal{G}_{\boldsymbol\theta}$-connections are exactly meromorphic $G$-connections. Also, the group $P_{\theta_x}$ is exactly $G$. Thus, $\boldsymbol\theta$-filtered Stokes $G$-local systems are exactly Stokes $G$-local systems when the weights $\boldsymbol\theta$ are trivial. Therefore, the two categories are equivalent by Theorem \ref{thm_Boal_equiv}.
\end{proof}

Now we move to general weights, and we first introduce the following relation of weights between de Rham side and Betti side,
\begin{center}
	\begin{tabular}{|c|c|c|}
		\hline
		\rule{0pt}{2.6ex} & \ \ \ de Rham\ \ \ & \ \ \ Betti\ \ \  \\[0.5ex]
		\hline
		\rule{0pt}{2.6ex} weights & $\beta$ & $\gamma$  \\[0.7ex]
		\hline
		\rule{0pt}{2.6ex} residues $\backslash$ monodromies & $\nabla_{\beta}$ & $M_\gamma$ \\[0.7ex]
		\hline
		\rule{0pt}{2.6ex} irregular types & $Q$ & $Q$ \\[0.7ex]
		\hline
	\end{tabular}
\end{center}
where $\nabla_\beta=s_\beta + Y_\beta$ is the Jordan decomposition, $\gamma = \beta - \frac{1}{2}(s_\beta + \bar{s}_\beta)$ and $M_\gamma = \exp(-2\pi \sqrt{-1}\nabla_\beta)$. 

\begin{cor}\label{cor_con_and_rep}
The category of merohoric $\mathcal{G}_{\boldsymbol\beta}$-connections on $X$ with irregular type $\boldsymbol{Q}$ and Levi factors of residues $\nabla_{\boldsymbol\beta}$ is equivalent to the category of $\boldsymbol\gamma$-filtered Stokes $G$-local systems on $X_{\boldsymbol{Q}}$ with Levi factors of formal monodromies $M_{\boldsymbol\gamma}$.
\end{cor}

\begin{proof}
Compared to the case of trivial weights, the formal monodromies of  $\boldsymbol\gamma$-filtered Stokes $G$-local systems are taken from parabolic subgroups $P_{\gamma_x}$, while the residues of merohoric $\mathcal{G}_{\boldsymbol\beta}$-connections are elements in the Lie algebra $\mathfrak{p}_{\beta_x}$. We have to check whether formal monodromies and residues are related naturally. Now we work around a puncture $x \in \boldsymbol{D}$. Given a merohoric $\mathcal{G}_{\beta}$-connection with irregular type $Q$ and residue $\nabla_\beta$, it can be regarded as a meromorphic $G$-connection with irregular type $Q$ and residue in $\mathfrak{p}_{\beta}$ by Corollary \ref{cor_canonical_form}. Then, it corresponds to a Stokes $G$-local system $\rho$ by Theorem \ref{thm_Boal_equiv}. Now we consider its residue $\nabla_\beta$. By Proposition \ref{prop_canonical_form_parah}, $\nabla_\beta = s_\beta + Y_\beta \in \mathfrak{h} \cap \mathfrak{p}_\beta$, and $[s_\beta, Y_\beta]=0$. Then, 
\begin{align*}
[Y_\beta,\gamma] = [Y_\beta, \beta - \frac{1}{2}(s_\beta + \bar{s}_\beta)] = [Y_\beta, \beta].
\end{align*}
This gives rise to
\begin{align*}
Y_\beta \in  \mathfrak{h}, \quad Y_\beta \in \mathfrak{p}_\gamma,
\end{align*}
i.e. $\exp(-2\pi \sqrt{-1}\nabla_\beta) \in H \cap P_\gamma$. Thus, the corresponding Stokes $G$-local system $\rho$ is actually a $\gamma$-filtered Stokes $G$-local system. Therefore, by Theorem \ref{thm_Boal_equiv}, Lemma \ref{lem_conn_and_rep_trivial} and Corollary \ref{cor_canonical_form}, the corollary follows directly.
\end{proof}

Note that we first fix a maximal torus $T$ and then introduce weights and corresponding parabolic subgroups. Therefore, $T \subseteq L_{\theta_x} \subseteq P_{\theta_x}$. If a collection of parabolic subgroups $\boldsymbol{P}=\{P_x, x \in \boldsymbol{D}\}$ is invariant and compatible with a $\boldsymbol\theta$-filtered Stokes $G$-local system, then the maximal torus $T$ (of $L_{\theta_x}$) is included in $P_x$ for any $x \in \boldsymbol{D}$. Since $P_x$ are conjugate to each other, $P_x=P_{x'}$ for any $x, x' \in \boldsymbol{D}$. In this case, all parabolic subgroups in $\boldsymbol{P}$ are the same. Thus, we modify the compatibility condition for filtered Stokes $G$-local systems as follows.
\begin{defn}
	Let $P$ be a proper parabolic subgroup of $G$, and let $\rho$ be a $\boldsymbol\theta$-filtered Stokes $G$-local system. We say that $P$ is \emph{invariant}, if $T \subseteq P$. An invariant proper parabolic subgroup $P$ is \emph{compatible} with $\rho$, if we have
	\begin{align*}
	\rho(\gamma) P \rho(\gamma)^{-1} = P
	\end{align*}
	for any path $\gamma$ in $X_{\boldsymbol{Q}}$ from $b_i$ to $b_j$ (for arbitrary $i,j$). 
\end{defn}

Since $P$ is a proper parabolic subgroup, its normalizer is exactly itself. Thus, if we have 
\begin{align*}
\rho(\gamma) P \rho(\gamma)^{-1} = P, 
\end{align*}
then $\rho(\gamma) \in P$. This basic property shows that if $P$ is \emph{invariant and compatible} with $\rho$, then all elements $A_i, B_i, C_{x}, h_{x}, S_{x,j}$ are included in $P$. Therefore, if there is a proper parabolic subgroup $P$ compatible with $\rho$, the representation $\rho$ naturally induces a filtered Stokes $P$-local system, which will be denoted by $\rho_P$.

\begin{defn}
	Let $\rho$ be a $\boldsymbol\theta$-filtered Stokes $G$-local system, and let $P$ be a parabolic subgroup of $G$ compatible with $\rho$. Taking a character $\chi: P \rightarrow \mathbb{C}^*$, we define the \emph{degree of a $\boldsymbol\theta$-filtered Stokes $G$-local system $\rho$} as 
	\begin{align*}
	\deg^{\rm loc} \rho(P,\chi) := \langle \boldsymbol\theta, \chi \rangle.
	\end{align*}
	A $\boldsymbol\theta$-filtered Stokes $G$-local system $\rho$ is of \emph{degree zero}, if for any character $\chi$ of $G$, we have $\langle \boldsymbol\theta, \chi \rangle=0$.
\end{defn}

\begin{defn}\label{defn_stab_betti}
A $\boldsymbol\theta$-filtered Stokes $G$-local system is \emph{$R$-stable} (resp. \emph{$R$-semistable}), if for
	\begin{itemize}
		\item any invariant and compatible proper parabolic subgroup $P \subseteq G$,
		\item any nontrivial anti-dominant character $\chi: P \rightarrow \mathbb{C}^*$, which is trivial on the center of $P$,
	\end{itemize}
	one has
	\begin{align*}
	\deg^{\rm loc} \rho(P,\chi) > 0  \quad  (\text{resp.} \geq 0). 
	\end{align*}
\end{defn}

\begin{lem}\label{lem_stokes_trivial}
When the weights $\boldsymbol\theta$ are trivial, a $\boldsymbol\theta$-filtered Stokes $G$-local system is $R$-stable if and only if it is irreducible. Thus, the $R$-stability condition in this case is equivalent to the stability condition considered in \cite[Theorem 9.4]{Boalch2014}.
\end{lem}

\begin{proof}
If the representation $\rho$ is irreducible, there is no proper parabolic subgroup $P$ compatible with $\rho$; this implies $\rho$ is $R$-stable. For the other direction, we suppose that $\rho$ is $R$-stable. If there exists a proper parabolic subgroup $P$ compatible with $\rho$, then we have
	\begin{align*}
	\deg^{\rm loc} \rho(P,\chi) = 0,
	\end{align*}	
for any character $\chi: P \rightarrow \mathbb{C}^*$ because the weights $\boldsymbol\theta$ are trivial. This contradicts the fact that $\rho$ is $R$-stable.
\end{proof}

For general weights, we introduce the following category:
\begin{itemize}
\item $\mathcal{C}_{\rm B}(X_{\boldsymbol{Q}},G,\boldsymbol{\gamma}, M_{\boldsymbol\gamma} )$: the category of $R$-stable $\boldsymbol\gamma$-filtered Stokes $G$-local systems on $X_{\boldsymbol{Q}}$ of degree zero such that the Levi factors of formal monodromies are $M_{\boldsymbol\gamma}$ around punctures.  
\end{itemize}
Given the following table
\begin{center}
	\begin{tabular}{|c|c|c|}
		\hline
		\rule{0pt}{2.6ex} &\ \ \ de Rham\ \ \ &\ Betti\  \\[0.5ex]
		\hline
		\rule{0pt}{2.6ex} weights & $\beta$ & $\gamma=\beta - \frac{1}{2}(s_\beta + \bar{s}_\beta)$  \\[0.7ex]
		\hline
		\rule{0pt}{2.6ex} residues $\backslash$ monodromies & $\nabla_{\beta}$ & $M_\gamma = \exp(-2\pi \sqrt{-1} \nabla_\beta)$ \\[0.7ex]
		\hline
		\rule{0pt}{2.6ex} irregular types & $Q$ & $Q$ \\[0.7ex]
		\hline
	\end{tabular}
\end{center}
let $(\mathcal{V},\nabla)$ be a merohoric $\mathcal{G}_{\boldsymbol\beta}$-torsor in $\mathcal{C}_{\rm dR}(X,\mathcal{G}_{\boldsymbol\beta}, \nabla_{\boldsymbol\beta}, \boldsymbol{Q} )$, and denote by $\rho$ the corresponding $\boldsymbol\gamma$-filtered Stokes $G$-local system on $X_{\boldsymbol{Q}}$ by Corollary \ref{cor_con_and_rep}. Given a proper parabolic subgroup $P$, if $P$ is compatible with $\rho$, then by Theorem \ref{thm_Boal_equiv} and Corollary \ref{cor_con_and_rep}, we obtain a merohoric $\mathcal{P}_{\boldsymbol\beta}$-torsor $(\mathcal{V}_{\varsigma},\nabla_{\varsigma})$, which is induced by a compatible reduction of structure group $\varsigma: X\rightarrow \mathcal{V}/\mathcal{P}_{\boldsymbol\beta}$. The argument also holds in the other direction.

\begin{rem}
    Although Theorem \ref{thm_Boal_equiv} is given for complex reductive groups, the correspondence involving reductive groups $G$ can be generalized to the case involving parabolic subgroups $P \subseteq G$ with the same argument as in \cite[Appendix]{Boalch2014}.
\end{rem}

\begin{lem}
	With the same setup as above, let $P$ be a proper parabolic subgroup compatible with $\rho$. Denote by $(\mathcal{V}_{\varsigma},\nabla_{\varsigma})$ the corresponding merohoric $\mathcal{P}_{\boldsymbol\beta}$-torsor. Let $\chi$ be a character of $P$, and denote by $\kappa$ the corresponding character of $\mathcal{P}_{\boldsymbol\beta}$. Then, we have
	\begin{align*}
	parh \deg \mathcal{V}(\varsigma,\kappa) = \deg^{\rm loc} \rho(P,\chi).
	\end{align*}
\end{lem}

\begin{proof}
    The proof is similar to \cite[Lemma 6.5]{Simp}. Denote by $(V,d',h)$ the corresponding metrized $G$-connection on $X_{\boldsymbol{D}}$ under the functor $\Xi_{\rm dR}$. Denote by $\sigma$ and $\chi$ the corresponding reduction of structure group and character respectively (see the proof of Proposition \ref{prop_ana=alg Higgs and conn}). Then, $\chi_* E_\sigma$ gives a line bundle on $X_{\boldsymbol{D}}$. The line bundle $\chi_* E_\sigma$ extends to a line bundle $L$ on $X$ under the functor $\Xi$ in \S\ref{subsect_func_Xi} and the connection $\chi_* d'_{\sigma}$ can be extended to a meromorphic connection $\nabla_L$ on the line bundle. By definition of the analytic degree, we have
    \begin{align*}
	    \deg^{\rm an} V(h,\sigma,\chi) = \deg L + \langle \boldsymbol\beta,\chi \rangle.
    \end{align*}
    Moreover,
    \begin{align*}
        \deg L = \langle {\rm Re}(s_{\boldsymbol\beta}), \chi \rangle
    \end{align*}
    by the residue theorem. Therefore,
    \begin{align*}
	    parh \deg \mathcal{V}(\varsigma,\kappa) & = \deg^{\rm an} V(h,\sigma,\chi) \\
	    & =  \deg L + \langle \boldsymbol\beta,\chi \rangle =  \langle \boldsymbol\beta,\chi \rangle + \langle {\rm Re}(s_{\boldsymbol\beta}), \chi \rangle \\
	    & = \langle \boldsymbol{\gamma}, \chi \rangle = \deg^{\rm loc} \rho(P,\chi).
    \end{align*}
\end{proof}

As a direct result, the stability conditions are also equivalent:
\begin{prop}
	Let $(\mathcal{V},\nabla)$ be an element in $\mathcal{C}_{\rm dR}(X,\mathcal{G}_{\boldsymbol\beta}, \nabla_{\boldsymbol\beta}, \boldsymbol{Q} )$. Denote by $\rho$ the corresponding $\boldsymbol\gamma$-filtered Stokes $G$-local system of degree zero on $X_{\boldsymbol{Q}}$. Then, $(\mathcal{V},\nabla)$ is $R_h$-stable if and only if $\rho$ is $R$-stable.
\end{prop}

Thus, we have the main result in this section.
\begin{thm}\label{thm_dR_Betti}
	The categories $\mathcal{C}_{\rm dR}(X,\mathcal{G}_{\boldsymbol\beta}, \nabla_{\boldsymbol\beta}, \boldsymbol{Q} )$ and $\mathcal{C}_{\rm B}(X_{\boldsymbol{Q}},G,\boldsymbol{\gamma}, M_{\boldsymbol\gamma} )$ are equivalent.
\end{thm}

\begin{rem}\label{rem_}
	In this remark, we consider the special case when all irregular types are trivial, i.e. $Q_x=0$ for all $x \in \boldsymbol{D}$. One would expect that the result reduces to the \emph{tame} case naturally. However, it does not cover the full correspondence in the tame case as studied in \cite{HKSZ}. The problem comes from the local calculation of meromorphic $G$-connections and the Riemann--Hilbert correspondence. If the irregular type $Q$ is not trivial, the last step in Remark \ref{rem_cano_form_proof} shows that a meromorphic $G$-connection with irregular type $Q$ is gauge equivalent to a meromorphic $G$-connection in canonical form under the action of $G(R)$, i.e. the holomorphic part vanishes. However, this is not true when $Q$ is trivial, and the holomorphic part may not vanish and can include nilpotent elements as the coefficients. This property thus results in the failure of the Riemann--Hilbert correspondence (as a one-to-one correspondence). Nonetheless, when the irregular type $Q$ is trivial, Boalch has established such a one-to-one correspondence by introducing parahoric objects \cite{Bo}. 
\end{rem}

\bigskip
\noindent\small{\textsc{Institut F\"ur Mathematik, Ruprecht-Karls-Universit\"{a}t Heidelberg}\\
		 Im Neuenheimer Feld 205, Heidelberg 69120, Germany}\\
\emph{E-mail address}:  \texttt{pfhwang@mathi.uni-heidelberg.de}

\bigskip
\noindent\small{\textsc{Department of Mathematics, South China University of Technology}\\
		381 Wushan Rd, Tianhe Qu, Guangzhou, Guangdong, China}\\
\emph{E-mail address}:  \texttt{hsun71275@scut.edu.cn}


\begin{thebibliography}{99}
    \bibitem{BaVa}
    Babbitt, D. G., Varadarajan, V. S.: Formal reduction theory of meromorphic differential equations: a group theoretic view. \emph{Pacific J. Math.} \textbf{109}, no. 1, 1-80 (1983).
    
	\bibitem{BS}
	Balaji, V., Seshadri, C. S.: Moduli of parahoric $\mathcal{G}$-torsors on a compact Riemann surface. \textit{J. Algebraic Geom.} \textbf{24}, no. 1, 1-49 (2015).

    \bibitem{BBMY}
    Bezrukavnikov, R., Boixeda Alvarez, P., McBreen, M., Yun, Z.: Non-abelian Hodge moduli spaces and homogeneous affine Springer fibers. arXiv: 2209.14695.
	
	\bibitem{BB}
	Biquard, O., Boalch, P.: Wild non-abelian Hodge theory on curves. \emph{Compos. Math.} \textbf{140}, no. 1, 179-204 (2004).
	
	\bibitem{BGM}
	Biquard, O., Garcia-Prada, O., Mundet i Riera, I.: Parabolic Higgs bundles and representations of the fundamental group of a punctured surface into a real group. \emph{Adv. Math.} \textbf{372}, 107305, 70pp. (2020).	
	
	\bibitem{Bo}
	Boalch, P.: Riemann-Hilbert for tame complex parahoric connections. \emph{Transform. Groups} \textbf{16}, 27-50 (2011).
	
	\bibitem{Boalch2012}
	Boalch, P.: HyperK\"ahler manifolds and nonabelian Hodge theory of (irregular) curves. arXiv: 1203.6607.

	\bibitem{Boalch2014}
	Boalch, P.: Geometry and braiding of Stokes data; fission and wild character varieties. \emph{Ann. of Math. (2)} \textbf{179}. no. 1, 310-365 (2014).

    \bibitem{Boalch2018}
    Boalch, P.: Wild character varieties, meromorphic Hitchin systems and Dynkin diagrams. In: Geometry and Physics. Vol. \textbf{2}. Oxford Univ. Press, Oxford. 433-454 (2018).

    \bibitem{BY15}
    Boalch, P., Yamakawa, D.: Twisted wild character varieties. arXiv: 1512.08091.
    
    \bibitem{ChenZhu}
    Chen, T. H., Zhu, X.: Non-abelian Hodge theory for algebraic curves in characteristic p. \emph{Geom. Funct. Anal.} \textbf{25}, 1706-1733 (2015).

    \bibitem{CDDNP2020}
    Chuang, W., Diaconescu, D.-E., Donagi, D., Nawata, S., Pantev, T.: Twisted spectral correspondence and torus knots. \emph{J. Knot Theor. Ramiﬁcations}. \textbf{29}, no. 6, 2050040 (2020).
	
	\bibitem{Corlette}
	Corlette, K.: Flat $G$-bundles with canonical metrics. \emph{J. Diff. Geom.} \textbf{28}, 361-382 (1988).
	
	\bibitem{DDP2018}
	Diaconescu, D.-E., Donagi, D., Pantev, T.: BPS states, torus links and wild character varieties. \emph{Comm. Math. Phys.} \textbf{359}, no. 3, 1027–1078 (2018).

    \bibitem{Donaldson}
    Donaldson, S. K.: Twisted harmonic maps and the self-duality equations. \emph{Proc. London Math. Soc. (3)} \textbf{55}, 127-131 (1987).

    \bibitem{GR2015}
    Garc\'ia-Raboso, A., Rayan, S.: Introduction to nonabelian Hodge theory: ﬂat connections, Higgs bundles and complex variations of Hodge structure. In: Calabi--Yau Varieties: Arithmetic, Geometry and Physics. \emph{Fields Inst. Res. Math. Sci.} \textbf{34}, 131-171 (2015).

	\bibitem{Herr}
	Herrero, A. F.: Reduction theory for connections over the formal punctured disc. arXiv: 2003.00008.
	
	\bibitem{Hit87}
	Hitchin, N. J.: The self-duality equations on a Riemann surface. \emph{Proc. London Math. Soc. (3)} \textbf{55}, 59-126 (1987).
	
	\bibitem{Huang2020}
	Huang, P.: Non-Abelian Hodge theory and related topics. \emph{SIGMA Symmetry Integrability Geom. Methods Appl.} \textbf{16}, Paper No. 029, 34pp (2020).
	
	\bibitem{HKSZ}
	Huang, P., Kydonakis, G., Sun, H., Zhao, L.: Tame nonabelian Hodge correspondence on curves. arXiv: 2205.15475.
 
	\bibitem{HS}
	Huang, P., Sun, H.: Moduli Spaces of Filtered (Stokes) G-local Systems on Curves. arXiv: 2304.09999.
 
	\bibitem{KSZ2parh}
	Kydonakis, G., Sun, H., Zhao, L.: Logahoric Higgs Torsors for a Complex Reductive Group. \emph{Math. Ann.} (2023).

    \bibitem{Li2019}
    Li, Q.: An introduction to Higgs bundles via harmonic maps. \emph{SIGMA Symmetry Integrability Geom. Methods Appl.} \textbf{15}, Paper No. 035, 30pp (2019).
	
	\bibitem{Malg}
	Malgrange, B.: \'Equations diff\'erentielles \`a coefficients polynomiaux. (French). Progress in Mathematics, \textbf{96}. \emph{Birkh\"auser Boston, Inc., Boston, MA} (1991).
	
	\bibitem{Mochi2011}
	Mochizuki, T.: Wild harmonic bundles and wild pure twistor D-modules. \emph{Ast\'erisque} no. \textbf{340}, 607 pp (2011).

    \bibitem{Mochi2021}
    Mochizuki, T.: Good wild harmonic bundles and good filtered Higgs bundles. \emph{SIGMA Symmetry Integrability Geom. Methods Appl.} \textbf{17}, Paper No. 068, 66pp (2021).
	
	\bibitem{Rama1975}
	Ramanathan, A.: Stable principal bundles on a compact Riemann surface. \textit{Math. Ann.} \textbf{213}, 129-152 (1975).
	
	\bibitem{Rama19961}
	Ramanathan, A.: Moduli for principal bundles over algebraic curves. I. \textit{Pro. Indian Acad. Sci. Math. Sci.} \textbf{106}, no. 3, 301-328 (1996).
	
	\bibitem{Rama19962}
	Ramanathan, A.: Moduli for principal bundles over algebraic curves. II. \textit{Pro. Indian Acad. Sci. Math. Sci.} \textbf{106}, no. 4, 421-449 (1996).
	
	\bibitem{Sab}
	Sabbah, C.: Harmonic metrices and connections with irregular singularities. \textit{Ann. Inst. Fourier (Grenoble)} \textbf{40}, no. 4, 1265-1291 (1999). 

    \bibitem{Seshadri}
	Seshadri, C. S.: Moduli  of  vector  bundles  on  curves  with  parabolic  structures.  \textit{Bull. Amer. Math. Soc.} \textbf{83}, no. 1, 124-126 (1977).

    \bibitem{Simp1988}
    Simpson, C. T.: Constructing variations of Hodge structure using Yang-Mills theory and applications to uniformization. \textit{J. Amer. Math. Soc} \textbf{1}, 867-918 (1988).
	
	\bibitem{Simp}
	Simpson, C. T.: Harmonic bundles on noncompact curves. \textit{J. Amer. Math. Soc.} \textbf{3}, no. 3, 713-770 (1990).
	
	\bibitem{Simp-naH}
	Simpson, C. T.: Nonabelian Hodge theory. \emph{Proc. Int. Congr. Math.}, Kyoto/Japan 1990, Vol. \textbf{I}, 747-756 (1991).
	
	\bibitem{Simp Higgs Local}
	Simpson, C. T.: Higgs bundles and local systems. \emph{Inst. Hautes \'{E}tudes Sci. Publ. Math.} \textbf{75}, 5-95 (1992).
	
	\bibitem{Simp2}
	Simpson, C. T.: Moduli of representations of the fundamental group of a smooth projective variety I. \textit{Inst. Hautes \'{E}tudes Sci. Publ. Math.} \textbf{79}, 47-129 (1994).
	
	\bibitem{Simp3}
	Simpson, C. T.: Moduli of representations of the fundamental group of a smooth projective variety II. \textit{Inst. Hautes \'{E}tudes Sci. Publ. Math.} \textbf{80}, 5-79 (1994).
	
	\bibitem{Su2019}
	Su, T.: Dual boundary complexes of Betti moduli spaces over the two sphere with one irregular singularity. arXiv: 2109.01645.
	
	\bibitem{vanSin}
	van der Put, M., Singer, M.: Galois theory of linear differential equations. Grundlehren der mathematischen Wissenschaften, \textbf{328}. \emph{Springer-Verlag, Berlin} (2003).
	
	\bibitem{Wasow}
	Wasow, W.: Asymptotic expansions for ordinary differential equations. Preprint of the 1976 edition. \emph{Dover Publications, Inc., New York} (1987).

    \bibitem{Witten2007}
    Witten, E.: Gauge theory and wild ramiﬁcation. \emph{Anal. Appl. (Singap.)} \textbf{6}, no. 4, 429-501 (2007).

	
	
\end{thebibliography}
\end{document}